\documentclass{article}
\usepackage[utf8]{inputenc}
\usepackage{amsmath,amssymb,amsthm, xcolor, enumitem, comment, todonotes}
\usepackage{url}
\usepackage[all]{xy}
\usepackage{graphicx}

\theoremstyle{definition}
\newtheorem{definition}{Definition}

\newtheorem{fact}[definition]{Fact}
\newtheorem{notation}[definition]{Notation}

\newtheorem{remark}[definition]{Remark}

\theoremstyle{plain}
\newtheorem{theorem}[definition]{Theorem}
\newtheorem{proposition}[definition]{Proposition}
\newtheorem{lemma}[definition]{Lemma}
\newtheorem{corollary}[definition]{Corollary}

\newcommand{\po}{\mathbb{P}}
\newcommand{\qo}{\mathbb{Q}}
\newcommand{\la}{\langle}
\newcommand{\ra}{\rangle}
\newcommand{\fr}{{}^{\frown}}
\newcommand{\name}{\dot}
\newcommand{\elem}{\prec}
\newcommand{\uhr}{\upharpoonright}
\newcommand{\mo}{\triangleleft}
\newcommand{\al}{\alpha}

\newcommand{\can}{\check}

\newcommand{\lam}{\lambda}
\newcommand{\power}{\mathcal{P}}
\newcommand{\A}{\mathfrak{A}}
\renewcommand{\S}{\mathcal{S}}

\DeclareMathOperator{\dom}{dom}
\DeclareMathOperator{\rng}{rng}
\DeclareMathOperator{\mc}{mc}
\DeclareMathOperator{\cf}{cof}
\DeclareMathOperator{\cp}{cp}

\DeclareMathOperator{\On}{On}
\DeclareMathOperator{\Coll}{Coll}
\DeclareMathOperator{\coll}{Coll}
\DeclareMathOperator{\otp}{otp}
\DeclareMathOperator{\pcf}{pcf}

\DeclareMathOperator{\dirlim}{dirlim}

\DeclareMathOperator{\low}{low}
\DeclareMathOperator{\high}{high}

\newcommand{\<}{\langle}
\renewcommand{\>}{\rangle}
\newcommand{\op}[1]{\operatorname{#1}}
\newcommand{\restr}{\upharpoonright}
\newcommand{\ult}{\op{Ult}}
\newcommand{\crit}{\op{crit}}
\newcommand{\gen}{\op{gen}}

\newcommand{\cof}{\op{cof}}
\newcommand{\card}{\op{card}}
\newcommand{\hull}{\op{Hull}}
\newcommand{\ptwimg}[2]{{#1}"\left[{#2}\right]}
\newcommand{\kleiner}{\mathord{<}}

\newcommand{\groesser}{\mathord{>}}
\newcommand{\groessergleich}{\mathord{\geq}}
\newcommand{\mitord}{\op{o}}
\newcommand{\length}{\op{lh}}
\newcommand{\eextend}{\trianglelefteq}
\newcommand{\ZFC}{\op{ZFC}}
\newcommand{\concat}{{}^\smallfrown}
\newcommand{\fnktsraum}[2]{{}^{#1}{}_{#2}}
\newcommand{\finsubsets}[1]{{\left[#1\right]}^{\kleiner\omega}}
\newcommand{\on}{\op{On}}
\newcommand{\ran}{\op{ran}}
\newcommand{\Th}{\op{Th}}
\newcommand{\id}{\op{id}}

\title{Approachable Free Subsets and Fine Structure Derived Scales}
\author{Dominik Adolf and Omer Ben-Neria}
\date{\today}

\begin{document}

\maketitle

\begin{abstract}
    Shelah showed that the existence of free subsets over internally approachable subalgebras follows from the failure of the PCF conjecture on intervals of regular cardinals. We show that a stronger property called the Approachable Bounded Subset Property can be forced from the assumption of a cardinal $\lambda$ for which the set of Mitchell orders $\{ o(\mu) \mid \mu < \lambda\}$ is unbounded in $\lambda$.     
    Furthermore, we study the related notion of continuous tree-like scales, and show that such scales must exist on all products in canonical inner models. We use this result, together with a covering-type argument, to show that the large cardinal hypothesis from the forcing part is optimal. 
\end{abstract}

\section{Introduction}

The study of set theoretic algebras has been central in many areas, with many applications to compactness principles, cardinal arithmetic, and combinatorial set theory. 

An algebra on a set $X$ is a tuple $\A = \la X,f_n\ra_{n<\omega}$ where $f_n: X^{k_n} \rightarrow X$ is a function.
A sub-algebra is a subset $M \subseteq X$ such that $f_n(x_0,\ldots,x_{k_n - 1}) \in M$ for all $(x_0,\ldots,x_{k_n - 1}) \in M^{k_n}$ and $n < \omega$. The set of sub-algebras of $\A$ is known as a club (in $\mathcal{P}(X)$). The characteristic function $\chi_M$ of $M$ is defined on the ordinals of $M$ by $\chi_M(\tau) = \sup(M \cap \tau)$. 

Shelah's celebrated bound in cardinal arithmetic (\cite{ShelahCardArith}) states that if $\aleph_\omega$ is a strong limit cardinal then 
$$
2^{\aleph_\omega} < \min\{ \aleph_{\omega_4},\aleph_{(2^{\aleph_0})^+}  \}.
$$

\noindent Starting from a supercompact cardinal, Shelah proved that for every $\alpha < \omega_1$, there exists a generic extension in which $2^{\aleph_\omega} = \aleph_{\alpha+1}$ (see \cite{GitikMagidorSCHrevosed}). It is a central open problem in cardinal arithmetic if  $2^{\aleph_\omega} \geq \aleph_{\omega_1}$ is consistent. 
A major breakthrough towards a possible solution is the work of Gitik (\cite{GitikSEFI},\cite{GitikSEFII}) on the failure of the PCF-conjecture. Shelah's PCF conjecture states that $|\pcf(A)| \leq |A|$ for every progressive\footnote{I.e., $\min(A) > |A|$.} set $A$ of regular cardinals. 
In \cite{ShelahFree}, Shelah has extracted remarkable freeness properties of sets over subalgrbras, from the assumption of $2^{\aleph_\omega} \geq \aleph_{\omega_1}$, or more generally, from the assumption $|\pcf(A)| > |A|$ for a progressive interval of regular cardinals $|A|$. 

\begin{definition}
Let $\A = \la X,f_n\ra_{n}$ be an algebra and $x \subset X$. 
We say that $x$ is \textbf{free} with respect to $\A$ if for every $\delta \in x$ and $n < \omega$, $\delta \not\in f_n`` (x\setminus\{\delta\})^{<\omega}$. \\
More generally, $x$ is free over a subalgebra $N \subseteq \A$ if for every $\delta \in x$ and $n < \omega$, 
 $\delta \not\in f_n`` (N \cup (x\setminus\{\delta\}))^{<\omega}$. \\
\end{definition}

\noindent A cardinal $\lambda$ has the Free Subset Property if every algebra $\A$ on $\lambda$ or a bigger $H_\theta$, has a free subset $x \subseteq \lambda$ which is cofinal in $\lambda$. A regular cardinal $\lambda$ with the Free Subset Property is Jonsson. 
Koepke \cite{Koepke} has shown that the free subset property at $\aleph_\omega$ is equiconsistent with the existence of a measurable cardinal. 
For a singular limit $\lambda$ of a progressive interval $|A|$, it is shown in \cite{ShelahFree} that if $|\pcf(A)| >|A|$ then $\lambda$ satisfies the Free Subset Property. 
In his PhD thesis (\cite{LuisThesis}), Pereira has isolated the notion of the Approachable Free Subset Property (AFSP) to play a critical role in the result from \cite{ShelahFree}.
The \text{Approachable Free Subset Property} for a singular cardinal $\lambda$ asserts that there exists some sufficiently large $H_\theta$, $\theta > \lambda$ and an algebra $\A$ on $H_\theta$ such that for every internally approachable substructure\footnote{See Definition \ref{def:IAstructure}} $N \elem \A$ with $|N| < \lambda$, there exists an infinite sequence of regular cardinal $\la \tau_i \mid i < \cf(\lambda)\ra \in N$ such that the set $x = \{ \chi_N(\tau_i) \mid i < \cf(\lambda)\}$ is free over $N$.\\
Pereira showed that Shelah's proof yields that if $\lambda$ is a limit of a progressive interval $A$ or regular cardinals and $|\pcf(A)| > |A|$ then the Approachable Free Subset Property holds at $\lambda$. 

Working with fixed sequences $\la \tau_n \mid n < \omega\ra$ of regular cardinal, we consider here the following version of this property.
\begin{definition}
The \textbf{Approachable Free Subset Property (AFSP)} with respect to $\la \tau_n\ra_n$ asserts that for every sufficiently large regular $\theta > \lambda = (\cup_n \tau_n)$ and for every internally approachable subalgebra $N \elem \A$, of an algebra $\A$ extending $(H_\theta,\in,\la \tau_n\ra_n)$, satisfying $|N| < \lambda$ there exists a cofinite set $x \subseteq \{ \chi_N(\tau_n) \mid n <\omega\}$ which is free over $N$. 
\end{definition}
By moving from one cardinal $\theta$ to $\theta' > \theta$ if needed, it is routine to verify the definition of AFSP with respect to a sequence $\la \tau_n \ra_n$ can be replaced with a similar assertion in which the requirement of ``every internally approachable $N$'' is replaced with `` for every internally approachable in some closed unbounded subset of $\power_\lambda(\A)$''.
Clearly, if AFSP holds with respect to a sequence $\la \tau_n \ra_n$ then AFSP holds with respect to the singular limit $
\lambda = \cup_n \tau_n$, as in the original definition of \cite{LuisThesis}.\\

\noindent The above mentioned results, suggest that AFSP can provide a path to possibly improving Shelah's bound, to  $2^{\aleph_\omega} < \aleph_{\omega_1}$. I.e., proving (in ZFC) that AFSP must fail at $\aleph_\omega$ (or AFSP fails w.r.t every subsequence $\la \tau_n \ra_n$ of $\{ \aleph_k \mid k < \omega\}$) would imply that $2^{\aleph_\omega} < \aleph_{\omega_1}$. 
To this end, Pereira (\cite{LuisThesis}) has isolated the notion of \textbf{tree-like scales}, as a potential tool of proving AFSP must fail.

\begin{definition}
Let $\la \tau_n \ra_{n < \omega}$ be an increasing sequence of regular cardinals. 
A scale\footnote{see Definition \ref{def:TLSandETLS} for the definition of a continuous scale} $\vec{f} = \la f_\alpha \mid \alpha < \eta\ra$ is a tree-like scale on $\prod_n \tau_n$ if for every $\alpha \neq \beta < \eta$ and $n < \omega$, 
 $f_\alpha(n+1) = f_\beta(n+1)$ implies $f_\alpha(n) = f_\beta(n)$.
\end{definition}

Pereira shows in \cite{Luis1} that the existence of a continuous tree-like scale on a product $\prod\limits_{n < \omega} \tau_n$ guarantees the failure of AFSP with respect to $\la \tau_n \ra_n$ (see also Lemma \ref{lem:ImplicationsBetweenTLSandAFSP}), and further proves that continuous tree-like scales, unlike other well-known types of scales, such as good scales, can exist in models with some of the strongest large cardinal notions, e.g. $I_0$-cardinals. Moreover, Cummings \cite{Cummings} proved that tree-like scales can exist above supercompact cardinals. These results show that as opposed to other well-known properties of scales such as good and very-good scales, which exhibit desirable "local" behaviour but cannot exist in the presence of certain large cardinals (\cite{CFM}),
the notion of continuous tree-like scales may coexist with the some of the strongest large cardinals hypothesis. \\

\noindent The consistency of the inexistence of a continuous tree-like scale on a product $\prod_n \tau_n$ of regular cardinal has been established by Gitik in 
\cite{Gitik}, from the consistency assumption of a cardinal $\kappa$ satisfying $o(\kappa) = \kappa^{++}+1$.
The argument makes a sophisticated use of the key features of Gitik's extender based Prikry forcing by a $(\kappa,\kappa^{++})$-extender.\footnote{E.g., on the fact that there are unboundedly many pairs $(\alpha,\alpha^*) \in [\kappa]^2$, sharing the same Rudin-Keisler projection map $\pi_{\alpha^*,\alpha}$.} 
Concerning the possible consistency of the  Approachable Free Subset Property, Welch (\cite{Welch}) has shown that AFSP with respect to a sequence $\la \tau_n\ra_n$ implies that the large cardinal assumption of Theorem \ref{thmone} holds in an inner model.\\

\noindent
It remained open whether AFSP with respect to some sequence $\la \tau_n \mid n< \omega\ra$ is consistent at all, and if so, whether its consistency strength is strictly stronger than the (seemingly) weaker property, of no continuous tree-like scale on $\prod_n \tau_n$. 
The current work answers both questions:

\begin{theorem}\label{thmone} 
  It is consistent relative to the existence of a cardinal $\lambda$ such that the set of Mitchell orders
  $\{ o(\mu) \mid \mu < \lambda\}$ is unbounded in $\lam$, that the Approachable Free Subset Property holds with 
  respect to some sequence of regular cardinals $\vec{\tau} = \la \tau_n \ra_n$. \\
  Moreover,the sequence $\tau_n$ can be made to be a subsequence of the first uncountable cardinals, in a model where $\lambda = \aleph_\omega$.
\end{theorem}

\begin{theorem}\label{thmthree}
    Let $\lambda$ be a singular cardinal of countable cofinality such that there is no inner model $M$ with $\lambda = \sup\{\mitord^M(\mu) \mid \mu < \lambda\}$. Let $\<\tau_n \mid n < \omega\>$ be a sequence of regular cardinals cofinal in $\lambda$. Then $\prod\limits_{n < \omega} \tau_n$ carries a continuous tree-like scale.
\end{theorem}

To achieve the proof of Theorem \ref{thmthree}, we establish a result of an independent interest, that the continuous tree-like scales naturally appear in fine-structural canonical inner models. Thus obtaining complementary result to aforementioned theorems by Pereira and Cummings, i.e. we know that no large cardinal property that can consistently appear in canonical inner models disproves the existence of products with continuous tree-like scales (e.g., Woodin cardinals).

\begin{theorem}\label{thmtwo}
   Let $\mathcal{M}$ be a premouse such that each countable hull has an $\omega$-maximal $(\omega_1 + 1)$-iteration strategy. Let $\lambda \in \mathcal{M}$ be a singular cardinal of countable cofinality. Let $\<\kappa_i : i < \omega\>$ be a sequence of regular cardinals cofinal in $\lambda$. Then $\prod\limits_{i < \omega} \kappa_i \slash J_{bd}$ carries a continuous tree-like scale.
\end{theorem}

Continuous tree-like scales on products of successor cardinals in $L$ where implicitly constructed by Donder, Jensen, and Stanly in \cite{DJS}.  In the course of proving Theorem \ref{thmone}, we establish the consistency of a principle stronger than AFSP, which we call the Approchable Bounded Subset Property. \\
Let $N$ be a subalgebra of $\A = \la H_\theta,f_n\ra_n$ and $\vec{\tau} = \la \tau_n \ra_n$ be an increasing sequene of cardinals.
Given a set $x \subseteq H_\theta$, we define $N[x]$ to be the $\A$-closure of the set $(x \cup N)$. 
We say that $N$ satisfies the \textbf{Bounded Appending Property} with respect to $\vec{\tau}$ if
for every $n_0 < \omega$, setting $x = \{ \chi_N(\tau_n) \mid n \neq n_0\}$ then the addition of $x$ to $N$ does not increase the supremum below $\tau_{n_0}$, namely $\chi_{N[x]}(\tau_{n_0}) = \chi_N(\tau_{n_0})$. 

\begin{definition}
The \textbf{Approachable Bounded Subset Property} holds with respect to $\la \tau_n \ra_n$ 
if for every sufficiently large regular $\theta > \lambda = (\cup_n \tau_n)$ and internally approachable subalgebra $N \elem \A$, of an algebra $\A$ extending $(H_\theta,\in,\la \tau_n\ra_n)$, that satisfies $|N| < \lambda$, then $N$ satisfies the bounded appending property with respect to a tail of $\la \tau_n \ra_n$.
\end{definition}

We show in Lemma \ref{lem:ImplicationsBetweenTLSandAFSP} ABSP with respect to a sequence $\la \tau_n\ra_n$ implies AFSP with respect to the same sequence, as well as the inexistence of a continuous \textbf{essentially tree-like} scale; a weakening of tree-like scale introduced by Pereira (see Definition \ref{def:TLSandETLS}). The proof of the forcing Theorem \ref{thmone}, stated above, goes through proving that ABSP is consistent with respect to a sequence of regulars $\la \tau_n \ra_n$. \\

\noindent
The following summarizes the main results of this paper:
\begin{corollary}
The following principles are equiconsistent:
\begin{enumerate}
    \item There exists a sequence of regular cardinals $\la \tau_n \mid n < \omega\ra$ for which the Approachable Bounded Subset Property holds.
    \item There exists a sequence of regular cardinals $\la \tau_n \mid n < \omega\ra$ for which the Approachable Free Subset Property holds.
    \item There exists a sequence of regular cardinals $\la \tau_n \mid n < \omega\ra$ for which the product $\prod_n \tau_n$ does not carry a continuous Tree-Like scale.
    \item There exists a cardinal $\lambda$ such that the set of Mitchell orders $\{o(\mu) \mid \mu < \lambda\}$ is unbounded in $\lambda$.
\end{enumerate}
\end{corollary}

\noindent \textbf{The paper is organized as follows:} The remainder of \textbf{this section} will be dedicated to discussing  preliminary material in PCF theory and the theory of inner models. \textbf{Section \ref{sec:forcing}} will dedicated to the forcing argument establishing the proof Theorem \ref{thmone}. In \textbf{Section \ref{sec:finestructure}} we discuss how to construct tree-like scales from the fine structure of canonical inner models. In \textbf{Section \ref{sec:CMandTLS}} we will use these fine structural scales to derive scales on products in $V$ using a covering-like argument. Finally, in \textbf{Section \ref{sec:open}} we finish with a list of open problems.\\

\noindent
\textbf{Acknowledgments:}\\
The work on this project was initiated following a suggestion by Assaf Rinot to study the consistency of the Approachable Free Subset Property. The authors are grateful for this suggestion and for supporting the first author during the academic year of 2018-2019 at Bar-Ilan University under a grant from the European Research Council (grant agreement ERC-2018-StG 802756). 
The initial idea for the inner model construction of continuous tree-like scale was conceived during the Berkeley conference on Inner Model theory in July 2019. The authors would like to thank Ralf Schindler and John Steel for organizing the meeting and creating the opportunity for this collaboration.
The first author would like to thank Grigor Sargsyan for his generous support and warm hospitality during the Spring of 2020 (NSF career award DMS-1352034). During that time the first author had the opportunity to travel to Pittsburgh. It was there that some significant improvements were made to the lower bound argument, and would like to thank James Cummings for the opportunity to present this research and insightful conversations.
The second author was partially supported by the Israel Science Foundation (Grant 1832/19). He would like thank Luis Pereira for insightful discussions on the subject and many valuable remarks on this paper, and to Spencer Unger and Philip Welch for many valuable comments and suggestions.

\subsection{Preliminaries}
For a set $X$ and a cardinals $\lam$, 
$\power_\lambda(X)$ denotes the collections of all subsets $a \subseteq X$ of size $|a| < \lam$.
$J_{bd}$ denotes the ideal of bounded subsets of $\omega$. 
Let $I$ be an ideal on $\omega$ and $f,g$ two functions from $\omega$ to ordinals. 
We write $f <_I g$ if $\{ n < \omega \mid f(n) \geq g(n)\} \in I$.
We write $f <^* g$ for $f <_{J_{bd}} g$. 

\subsubsection{Continuous and Tree-Like Scales}
Let $\la \tau_n \mid n < \omega\ra$ be a sequence of ordinals of strictly increasing cofinalities. 
A sequence of functions $\vec{f} =\la f_\alpha \mid \alpha < \eta\ra \subseteq \prod_n \tau_n$ of a regular length $\eta$, is a \textbf{pre-scale} in $(\prod_n \tau_n, <_I)$ if $\vec{f}$ is strictly increasing in the ordering $<_I$. A prescale is a \textbf{scale} if it is cofinal in $\prod_n \tau_n$.
As we focus on $J_{bd}$ from this point forward, we will frequently say that $\vec{f}$ is a (pre-)scale in $\prod_n \tau_n$, without mentioning the ideal $J_{bd}$.

\begin{definition}\label{def:TLSandETLS}
Suppose that $\vec{f} = \la f_\alpha \mid \alpha < \eta\ra$ is a (pre-)scale in $\prod_n \tau_n$.
\begin{enumerate}
    \item $\vec{f}$ is \textbf{continuous} if for every limit ordinal $\delta < \eta$ of uncountable cofinality, the sequence $\vec{f}\uhr\delta$ is $<^*$-cofinal in $\prod_n f_\delta(n)$. 
    
    \item $\vec{f}$ is \textbf{Tree-like} if for every $\alpha \neq \beta < \eta$ and $n < \omega$, 
    if $f_\alpha(n+1) = f_\beta(n+1)$ then $f_\alpha(n) = f_\beta(n)$.
    
    \item $\vec{f}$ is  \textbf{Essentially Tree-like} if for every $n < \omega$ and $\mu \in [\tau_n,\tau_{n+1})$
    the set 
    \[
  \{ \mu' < \tau_n \mid \exists \beta < \eta, f_\beta(n+1) = \mu \text{ and } f_\beta(n) = \mu'\}\]
    is nonstationary in $\tau_n$.
\end{enumerate}
\end{definition}

\noindent
If a product $\prod_n \tau_n$ carries a scale, it is not difficult to find another scale on it with the tree-like property (see Pereira \cite{LuisThesis}), but such a scale need not be continuous. 

\subsubsection{Internally Approachable Structures and related Principles}

Considering notions such as the Approachable Free Subset Property or the Approachable Bounded Subset Property with respect to subalgebras of Algebras $\A = (\theta,f_n)_n$, 
there is no harm in replacing the domain $\theta$ with another set of the same size, such as $H_\theta$ in cases relevant to us, and adding more structure to the algebra.
Therefore, from this point on, we will only restrict ourselves to set theoretic algebras $\A$ of the form 
$\A = (H_\theta,\in,f_n)_n$, which extend the model $(H_\theta,\in)$ in the language of set theory, and include Skolem functions. In particular, a subalgebra $N \elem \A$ will always be an elementary substructure. 

This allows us to reformulate our notion of freeness. Assuming\footnote{we will always be abe to assume so} the algebra $\A$ is rich enough to satisfy a fraction of ZFC\footnote{specifically, the Replacement property}, and $N \subseteq \A$ is sufficiently closed so it is an elementary substructure $N \elem \A$, then the fact that a set $x$ is free over $N$ is equivalent to 
having that for every $\delta \in x$ and a function $f \in N$, $\delta \not\in f`` (x\setminus\{\delta\})^{<\omega}$. \\ 

The notion of internally approachable structures was formally introduced in \cite{ForemanMagidorShelah}. We refer the reader to  \cite{ForemanMagidor-DefinableCounterOfCH} for further exposition. 
The definition below is similar to the standard ones, with the addition that here, we will focus on internally approachable unions of uncountable cofinality. 

\begin{definition}\label{def:IAstructure}
An elementary subalgebra (substructure) $N \elem \A$ of an algebra $\A=  (H_\theta;\in, f_n)_n$  is said to be internally approachable of length $\rho$ if 
$N = \bigcup_{i<\rho} N_i$ is a union of a sequence 
$\vec{N} = \la N_i \mid i < \rho\ra$ of elementary subalgebras $N_i \elem N$, and for every $j < \rho$, 
$\vec{N}\uhr j = \la N_i \mid i < j\ra$ belongs to $N$.\\

We say that $N \elem \A$ is \textbf{internally approachable} if it is internally approachable of length $\rho$ for some $\rho$ of uncountable cofinality $\cf(\rho) > \aleph_0$.
\end{definition}

\begin{notation}
Let $N \elem (H_\theta;\in)$ for a regular cardinal $\theta$.
\begin{itemize}
    \item For every regular cardinal $\tau \in N$, define $\chi_N(\tau) = \sup(N \cap \tau)$.
    \item Given a sequence $\vec{\tau} = \la \tau_n \mid n < \omega\ra  \subseteq N$ define the function 
    $\chi^{\vec{\tau}}_N \in \prod_n \tau_n$ by $\chi^{\vec{\tau}}_N(n) = \chi_N(\tau_n)$ if the last ordinal is strictly smaller than $\tau_n$, and $0$ otherwise.
\end{itemize}
\end{notation}

The following folklore result connects continuous scales with characteristic functions of internally approachable structures. We include a proof for completeness. 
\begin{lemma}\label{lem:IAcharFunc}
  Suppose that $\vec{\tau} = \la \tau_n  \mid n < \omega\ra\in N$ is a strictly increasing sequence of regular cardinals for which $\prod_n \tau_n$ carries a continuous scale $\vec{f} = \la f_\alpha \mid \alpha < \eta\ra$. 
  For every $N \elem H_\theta$ which is internally approachable of size
  $|N| < \bigcup_n \tau_n$, with $\vec{f} \in N$, if $\delta = \chi_N(\eta)$ then $\chi_N^{\vec{\tau}}(n) = f_\delta(n)$ for 
  all but finitely many $n < \omega$.
\end{lemma}
\begin{proof}
Let $\vec{N} = \la N_i \mid i < \rho\ra$ be a sequence witnessing $N = \cup_i N_i$ is internally approachable of length $\rho$ which has uncountable cofinality. 
Since $\vec{f}$ is continuous, it suffices to show that $\vec{f}\uhr\delta$ is $<^*$-cofinally interleaved with the functions in $\prod_n \chi_N^{\vec{\tau}}(n)$ to prove that $\chi_N^{\vec{\tau}}(n) = f_\delta(n)$ for almost all $n < \omega$.
First,for every $f_\alpha \in \vec{f}\uhr\delta$ there exists some $\beta\in N \cap \delta$ so that $\alpha < \beta$, and thus
$f_\alpha <^* f_\beta$. But $f_\beta \in N$ since $\beta\in N$, which means that $f_\beta \in \prod_n \chi_N^{\vec{\tau}}(n)$.
Next, fix $g \in \prod_n \chi_N(\tau_n)$. We show that $g <^* f_\alpha$ for some $\alpha < \delta$. To this end,  $N = \bigcup_{i < \rho}N_i$ guarantees that for each $n < \omega$ there is $i < \rho$ such that $g(n) < \chi_{N_i}(\tau_n)$. Since $\cf(\rho) > \aleph_0$ there is $i < \rho$ such that $g(n) < \chi_{N_i}(\tau_n)$ for all $n$, and in particular, 
$g <^* \chi_{N_i}^{\vec{\tau}}$. Since $\vec{f} \in N$ is $<^*$-cofinal in $\prod_n\tau_n$ and $\chi_{N_i}^{\vec{\tau}} \in N$,
 there is some $\alpha \in N \cap \eta \subseteq \delta$ so that $\chi_{N_i}^{\vec{\tau}} <^* f_\alpha$, and thus $g <^* f_\alpha$.
\end{proof}

\begin{lemma}\label{lem:IApressingdown}
Let $\lambda < \theta$ be cardinals with $\theta$ regular, and $\mo$ be a well-ordering of $H_\theta$.
Suppose that $\S \subseteq \power_{\lambda}(H_\theta)$ is a stationary set of 
  internally approachable structures $N \elem (H_\theta;\in,\mo)$, and $X \in H_\theta$ is a set which belongs to all $N \in \S$, and
 satisfies that $|X^\omega| \leq \eta$ is a regular cardinal and $\rho^{\aleph_0} < \eta$ for every cardinal $\rho < \lambda$.
 Then, for every assignment which maps each $N \in \S$ to a countable sequence
 $\la x^N_n \mid n < \omega\ra\in X^\omega$ which is contained in $N$, 
 there exists a stationary subset $S^* \subseteq \eta$ and a constant sequence $\la x_n \mid n < \omega\ra$ such that for every 
 $\delta \in S^*$ there is $N \in \S$ satisfying $\chi_N(\eta) = \delta$ and $\la x^N_n \ra_n = \la x_n\ra_n$.
\end{lemma}
\begin{proof}
Let $\la \vec{x}^\alpha \mid \alpha < \eta\ra$ be the $\mo$-least enumeration of $X^\omega$ in $H_\theta$, where each $\vec{x}^\al$ is of the form  $\la x^\al_n \mid n < \omega\ra$. 
For each $N \in \S$ let $\alpha_N < \eta$ be such that $\la x^N_n \mid n < \omega\ra = \vec{x}^{\alpha_N}$. Note that $\alpha_N$ need not be a member of $N$ since $\la x^N_n \mid n < \omega\ra$ need not. 
Since each $N \in \S$ is the union of a sequence $\la N_i \mid i < \rho\ra$ with $\cf(\rho)> \aleph_0$, and $\la x^N_n \mid n < \omega \ra \subseteq N$ there is some $i < \rho$ so that $\la x^N_n \mid n < \omega\ra \subset N_i$, and thus 
$\la x^N_n \mid n < \omega\ra \in (X \cap N_i)^\omega \in N$. Moreover, as $|X \cap N_i|<\lambda$, we have that $|(X \cap N_i)^\omega| < \eta$, and therefore there exists  
some $\beta_N \in \eta \cap N$ so that $(X \cap N_i)^\omega \subset \la \vec{x}^\alpha \mid \alpha < \beta_N\ra$. We conclude that, $\alpha_N < \beta_N <  \chi_N(\eta)$.
Next, define $S = \{ \chi_N(\eta) \mid N \in \S\}$. $S \subseteq \eta$ is stationary, and by choosing for each $\delta \in S$ a specific structure  $N_\delta \in \S$ with $\delta = \chi_N(\eta)$, we can form a pressing down assignment taking each $\delta \in S$ to $\alpha_{N_\delta} < \delta$. 
Let $\alpha^* < \eta$ and $S^* \subseteq S$ be so that $\alpha_{N_\delta} = \alpha^*$ for all $\delta \in S^*$. The claim follows for $S^*$ and 
$\la x_n \mid n < \omega\ra = \vec{x}^{\alpha^*}$. 
\end{proof}

Let $\vec{\tau} = \la \tau_n \mid n < \omega\ra$ be an increasing sequence of regular cardinals, $\lam = \cup_n \tau_n$, and $\theta > \lambda^+$  regular. 
A set $C \subseteq \power_\lambda(H_\theta)$ is a closed unbounded set if it contains all elementary substructures $M \elem \A$ of size $|M| < \lambda$ of some algebra $\A = (H_\theta,\in,f_n)_n$ on $H_\theta$.
We reformulate the definitions of Approachable Free Subset Property and Approachable Bounded Subset Property from the introduction. 

\begin{definition}\label{def:AF/BSP}
\begin{enumerate}
     \item Let $F : [\lambda]^{<\omega} \to \lambda$ be a function. We say that a subset $X \subseteq \lambda$ is \textbf{free with respect to $F$} if for every $\gamma \in X$, $\gamma \not\in F[X\setminus \{\gamma\}]^{<\omega}$.
     
     \item The \textbf{Approachable Free Subset Property (AFSP)} with respect to $\vec{\tau}$ asserts that there exists a closed unbounded set $C \subseteq \power_\lam(H_\theta)$ of structures $N \elem (H_\theta;\in)$ so that 
     for every internally approachable structure $N \in C$ there exists some $m < \omega$ such that the set $\{ \chi_N(\tau_n) \mid m \leq n < \omega\}$ is free with respect to every function $F \in N$
     
      \item The \textbf{Approachable Bounded Subset Property (ABSP)}with respect to $\vec{\tau}$ asserts that there exists a closed unbounded set $C \subseteq \power_\lam(H_\theta)$ of structures $N \elem (H_\theta;\in)$ so that for every internally approachable structure $N \in C$  there exists some $m < \omega$ 
      such that for every $F\in N$, $F : [\lambda]^k \to \lambda$ of finite arity $k< \omega$,
      and distinct numbers $d, d_1 , d_2 , \dots , d_k \in \omega\setminus m$, if 
      \[F\left(\chi_N(\tau_{d_1}),\dots \chi_N(\tau_{d_k})\right) < \tau_d\]      
      then 
      \[F\left(\chi_N(\tau_{d_1}),\dots \chi_N(\tau_{c_d})\right) < \chi_N(\tau_{d}).\]
\end{enumerate}
\end{definition}

To see that the formulations in Definition \ref{def:AF/BSP} are equivalent to the ones given in the introduction,  note that 
if $\theta > \lambda = \cup_n \tau_n$ is the first for which that there exists a club $C \subseteq \power_\lambda(H_\theta)$ which is definable in $\vec{\tau}$ consisting of subalgebra $M \subseteq \A = (H_\theta,\in,f_n)_n$, then for every $\theta' > \theta$
and $M' \elem H_{\theta'}$, if $\vec{\tau} \in M'$ then $\theta,C \in M'$ and $M' \cap C \in C$. 

\begin{lemma}\label{lem:ImplicationsBetweenTLSandAFSP}
Suppose that $\vec{\tau} = \la \tau_n \mid n < \omega\ra$ is an increasing sequence of regular cardinals. 
\begin{enumerate}

  \item If there is no continuous essentially tree-like scale on 
  $\prod_n \tau_n$ then there is no continuous tree-like scale on $\prod_n \tau_n$.
  
  \item AFSP w.r.t $\vec{\tau}$ implies that there is no continuous tree-like scale on $\prod_n \tau_n$.
  
    \item ABSP w.r.t $\vec{\tau}$ implies both \\
    (i) AFSP w.r.t  $\vec{\tau}$, and \\
    (ii) there is no continuous essentially tree-like scale on $\prod_n \tau_n$.
\end{enumerate}
\end{lemma}
\begin{proof}
\begin{enumerate}
    \item  This is an immediate consequence of the definitions of an essentially tree-like scale and a tree-like scale on $\prod_n \tau_n$.

\item 
We prove the contrapositive statement, that if there exists a continuous tree-like scale on $\prod_n \tau_n$ then AFSP fails with respect to $\vec{\tau}$.
Suppose that $\vec{f}$ is a continuous tree-like scale on $\prod_n \tau_n$. 
Since $\vec{f}$ is tree-like, we can assign to it a function $F : \lambda \to \lambda$, $\lambda= \cup_n \tau_n$, defined as follows:
For every $n < \omega$ and $\mu$, $\tau_n \leq \mu < \tau_{n+1}$, define
\[
F(\mu) = 
\begin{cases}
f_\alpha(n) &\mbox{if } \mu = f_{\alpha}(n+1) \text{ for some } \al < \eta\\
0 &\mbox{otherwise.}
\end{cases}
\]
$F(\mu)$ is well defined, i.e., does not depend on the choice of $\alpha$ such that $\mu = f_\alpha(n+1)$, since $\vec{f}$ is tree-like.
It is clear from the definition of $F$ that for every $\delta < \eta$ and $n < \omega$, $F(f_\delta(n+1)) = f_\delta(n)$.
Now, if $C \subseteq \power_{\lambda}(H_\theta)$ is a closed unbounded subset, 
 $N \in C$ is an internally approachable structure  with $F \in N$, and $\delta = \chi_N(\eta)$, then
$\chi_N^{\vec{\tau}}(n) = f_\delta(n)$ for all but finitely many $n < \omega$. 
Hence, for all but finitely many $n < \omega$, $F(\chi_N(\tau_{n+1})) = \chi_N(\tau_n)$, which means that $\{\chi_N(\tau_{n+1}),\chi_N(\tau_{n})\}$ is not free with respect to $F \in N$. Since $C$ was an arbitrary closed and unbounded subset, 
AFSP with respect to $\vec{\tau}$ fails. 

\item 
The fact that ABSP implies AFSP is immediate from the definition of the two properties. 
To show that ABSP w.r.t $\vec{\tau}$ implies that there is no continuous scale on $\prod_n \tau_n$ which is essentially tree-like, 
we prove the contrapositive statement. 
Suppose that $\la f_\alpha \mid \alpha < \eta\ra$ is a continuous essentially tree-like scale on a product $\prod_n \tau_n$. Then by Definition 
\ref{def:TLSandETLS} for every $n < \omega$, there is a function $C_n : \tau_{n+1} \to \power(\tau_n)$ so that for every $\mu < \tau_{n+1}$, 
$C_n(\mu)$ is a closed and unbounded subset of $\tau_n$ which is disjoint from $\{ \mu_n < \tau_n \mid \exists \beta < \eta, f_\beta(n+1) = \mu \text{ and } f_\beta(n) = \mu_n\}$.
Let $C$ be any club of elementary substructures of $(H_\theta;\in)$. Take an internally approachable substructure $N \in C$ and of size $|N| < \lambda = \cup_n \tau_n$, so that both $\la \tau_n \mid n< \omega\ra$ and $\la C_n \mid n < \omega\ra$ belong to $N$.
Define $\delta = \chi_N(\eta)$ and let $m < \omega$ so that $f_\delta(n) = \chi_N(\tau_n)$ for all $n \geq m$.
Fixing $n \geq m$ and examining the elementary extension 
$N' = N[\{f_\delta(n+1)\}] = \{ F(f_\delta(n+1)) \mid F \in N\} \elem (H_\theta;\in)$ of $N$,
we have that $C_n(f_\delta(n+1)) \in N'$ since $C_n \in N$. 
Now, as $C_n(f_\delta(n+1)) \subseteq \tau_n$ is closed unbounded, we must have that $\chi_{N'}(\tau_n) \in C_n(f_\delta(n+1))$. However $\chi_N(\tau_n) = f_\delta(n) \not\in C_n(f_\delta(n+1))$ by the definition of $C_n$. This implies that $\chi_{N'}(\tau_n) > \chi_N(\tau_n) = f_\delta(n)$, which in turn, implies that $F(f_\delta(n+1)) > f_\delta(n)$ for some $F \in N$. Since $N \in C$ where $C$ is an arbitrary closed unbounded collection, and $n$ is an arbitrarily large finite ordinal, we conclude that 
ABSP fails with respect to $\la \tau_n \mid n < \omega\ra$.
\end{enumerate}
\end{proof}   

\subsection{Fine structure primer}

\subsubsection{Ultrafilters}

We shall take our fine structure from \cite{FSIT}. Our result almost certainly also applies to different forms of fine structure such as the fine structure theory of \cite{NewFS}, in fact, the proof of Theorem \ref{thmtwo} in particular would be greatly simplified, but at the cost of significantly complicating the arguments in the core model part of this paper. As there is currently no account of the covering lemma for $\lambda$-indexing, we think it prudent to choose Mitchell-Steel mice at this time. We don't use \cite{Zeman}, as $\lnot O^\P$ is much too strong a limitation for this section. (While technically Mitchell and Steel operate under the assumption of $\lnot M^\#_1$ in \cite{FSIT}, it is well understood by now that their fine structure theory functions well past this point.)

 For our purposes an extender $F$ is a directed system of ultrafilters $\{(a,X)\vert a \in \finsubsets{\length(F)}, X \subset \left[\crit(F)\right]^{\vert a\vert}\}$ as described in \cite[p. 384]{Jech}. The individual ultrafilters will be denoted as $F_a := \{ X \subset \crit(F)^{\vert a\vert} \vert (a,X) \in F\}$. For $a \subset b$ and $f$ a function with domain $\left[\crit{F}\right]^{\vert a\vert}$, we let $f^{a,b}$ be the function with domain $\left[\crit{F}\right]^{\vert b\vert}$ determined by $f^{a,b}(\bar{b}) = f(\bar{a})$ where $\bar{a}$ is the unique subset of $\bar{b}$ determined by the type of $a$ and $b$. This gives rise to an embeddings from $\ult(\mathcal{M},F_a)$ into $\ult(\mathcal{M},F_b)$. The direct limit along those embeddings is the extender ultrapower $\ult(\mathcal{M},F)$, elements of which we will present as pairs $\left[f,a\right]^\mathcal{M}_F$ where $f \in \mathcal{M}$ is a function with domain $\left[\crit(F)\right]^{\vert a\vert}$ and $a \in \finsubsets{\length(F)}$. The direct limit map shall be denoted $\iota^\mathcal{M}_F: \mathcal{M} \rightarrow \ult(\mathcal{M},F)$. We will generally omit the superscript in this notation. This should not lead to confusion. Note that we will later form ultrapowers where some functions involved in the construction are not elements of the structure but merely definable over it.

$\beta < \length(F)$ is a generator of $F$ if it cannot be represented as $\left[f,a\right]_F$ for any $f \in \fnktsraum{\crit(F)}{\crit(F)} \cap \mathcal{M}$ and $a \in \finsubsets{\beta}$, i.e. $\{b \cup \{\xi\} \vert f(b) = \xi\} \notin F_{a \cup \{\beta\}}$. Let $\gen(F)$ denote the strict supremum of the generators of $F$. Also let $\nu(F) = \max\{\gen(F),(\crit(F)^+)^\mathcal{M}\}$.

For a subset $A$ of $\alpha$ we will write $F \restr A := \{ (a,X) \in F \vert a \subset A\}$. We will consider this an extender, forming ultrapowers etc, even if $A$ is not an ordinal. Let $\eta < \alpha$ be such that $\eta = \gen(F \restr \eta)$, then the trivial completion is the $(\crit(F),(\eta^+)^{\ult(\mathcal{M};F \restr \eta)})$-extender derived from $\iota_{F \restr \eta}$.

\subsubsection{Premice}

A potential premouse is a structure of the form $\mathcal{M} = \<J^{\vec{E}}_\alpha; \in, \vec{E}, F\>$ where $J^{\vec{E}}_\alpha$ is a model constructed from a sequence of extenders $\vec{E}$ using the Jensen hierarchy. For $\beta \leq \alpha$ we define $\mathcal{M} \vert \beta := (J^{\vec{E} \restr \beta}_\beta;\in, \vec{E}\restr \beta,\vec{E}_\beta)$ and $\mathcal{M}\vert\vert\beta := (J^{\vec{E}\restr\beta}_\beta;\in,\vec{E}\restr\beta)$. (The difference between the two notations lies in including a top predicate.) If $\mathcal{N}$ is of one of the above forms then we write $\mathcal{N} \eextend \mathcal{M}$ and say $\mathcal{N}$ is an initial segment of $\mathcal{M}$.  

$\vec{E}$ must be good, i.e. it has the following properties:

\begin{itemize}
    \item[(Idx)] for all $\beta < \alpha$ if $\vec{E}_\beta \neq \emptyset$, then $\beta = (\nu(\vec{E}_\beta)^+)^{\ult(\mathcal{M} \vert \beta; \vec{E}_\beta)}$;
    \item[(Coh)] for all $\beta < \alpha$ if $\vec{E}_\beta \neq \emptyset$, then $\mathcal{M} \vert\vert \beta = \ult(\mathcal{M}\vert\beta;\vec{E}_\beta) \vert\beta$;
    \item[(ISC)] for all $\beta < \alpha$ if $\vec{E}_\beta \neq \emptyset$, then for all $\eta < \alpha$ such that $\eta = \gen(\vec{E}_\beta \restr \eta)$ the trivial completion of $\vec{E}_\beta \restr \eta$ is on $\vec{E}$ or $\vec{E}_\eta \neq \emptyset$ and it is on $\iota_{\vec{E}_\eta}(\vec{E})$.
\end{itemize}

Note that $\vec{E}_\beta$ measures exactly those subsets of its critical point that are in $\mathcal{M}\vert\vert\beta$ for any $\beta < \alpha$ such that $\vec{E}_\beta \neq \emptyset$. $F$ the top extender must be such that $\vec{E}\concat F$ remains good. $F$ can be empty in which case $\mathcal{M}$ is called passive, otherwise $\mathcal{M}$ is active.

To an active potential premouse we associate three constants: $\mu^{\mathcal{M}}$ the critical point of the top extender; $\nu^{\mathcal{M}}$ the strict supremum of the generators of $\mathcal{M}$'s top extender or $((\mu^\mathcal{M})^+)^\mathcal{M}$ whichever is larger; $\gamma^\mathcal{M}$ the index of the longest initial segment of $\mathcal{M}$'s top extender (if it exists).

We distinguish three different types of active potential premouse: $\mathcal{M}$ is active type I if $\nu^\mathcal{M} = (\mu^{\mathcal{M},+})^\mathcal{M}$; $\mathcal{M}$ is active type II if $\nu^\mathcal{M}$ is a successor ordinal; $\mathcal{M}$ is a active type III if it is neither type I or type II, i.e. the set of the generators of $\mathcal{M}$'s top extender has limit type. 

\subsubsection{Fine structure}

The big disadvantage of Mitchell-Steel indexing is that we cannot deal directly with definability over $\mathcal{M}$, but instead need to work with an amenable code of our original structure. The exact nature of this coding is dependant on the type of $\mathcal{M}$. We will take inspiration from \cite{JohnHST} and use a uniform notation $\mathcal{C}_0(\mathcal{M})$ for this code. 

If $\mathcal{M} := \<\vert\mathcal{M}\vert;\in,\vec{E},F\>$ is an active potential premouse of type I or II, we will define an alternative predicate $F^{c}$ coding the top extender $F$: $F^c$ consists of tuples $(\gamma,\xi,a,X)$ such that $\xi \in \left(\mu^\mathcal{M},(\mu^{\mathcal{M},+})^\mathcal{M}\right)$ and $\gamma \in \left(\nu(F),\on \cap \vert\mathcal{M}\vert\right)$ is such that $(F \cap (\finsubsets{\nu(F)} \times \mathcal{M}\vert\vert\xi)) \in \mathcal{M}\vert\vert\gamma$, and $(a,X) \in (F \cap (\finsubsets{\gamma} \times \mathcal{M}\vert\vert\xi))$. The point is that $F^c$ is amenable. We let $\mathcal{C}_0(\mathcal{M}) := \<\vert\mathcal{M}\vert; \in, \vec{E},F^c\>$.

If $\mathcal{M}$ on the other hand is active type III we have to make bigger changes. In the language of \cite{FSIT} we have to ``squash", that is remove ordinals from the structure. (This is to ensure that the initial segment condition is preserved by iterations.) We let $\mathcal{C}_0(\mathcal{M}) := \< J^{\vec{E}}_{\nu(F)};\in,\vec{E} \restr \nu(F),F \restr \nu(F)\>$.

We then define $r\Sigma_1$-formulae to be $\Sigma_1$ over $\mathcal{C}_0(\mathcal{M})$, and $r\Sigma_{n + 1}$-formulae to be $\Sigma_1$ in a predicate coding an appropriate segment of the $r\Sigma_n$-theory of $\mathcal{C}_0(\mathcal{M})$. We will let $\Th^\mathcal{M}_n(\alpha,q) := \{ (\lceil \phi \rceil,b)\vert \phi \text{ is } r\Sigma_n, b \in \finsubsets{\alpha}, \mathcal{C}_0(\mathcal{M}) \models \phi(b,q)\}$.

Projecta can then be defined relative to these formulas, i.e. $\rho_{n+1}(\mathcal{N})$ is the least ordinal such that some $r\Sigma_{n+1}$-definable (in parameters) subset of it is not in $\mathcal{C}_0(\mathcal{M})$. $\rho_0(\mathcal{M}) = \on \cap \mathcal{C}_0(\mathcal{M})$ (which might be smaller than $\on \cap \mathcal{M}$).

As usual we define $p_{n+1}(\mathcal{M})$, the $(n+1)$-th standard parameter, to be the lexicographically least $p \in \finsubsets{\on \cap \mathcal{C}_0(\mathcal{M})\slash \rho_{n+1}(\mathcal{M})}$ that defines a missing subset of $\rho_{n+1}(\mathcal{M})$.

We can also define canonical $r\Sigma_{n+1}$-Skolem function allowing us to form $\hull^\mathcal{M}_{n+1}(A)$ given a subset $A$ of $\mathcal{C}_0(\mathcal{M})$. Note that while our notation makes it look like a hull of $\mathcal{M}$ it is a substructure of $\mathcal{C}_0(\mathcal{M})$ not $\mathcal{M}$.

We say $\mathcal{M}$ is $n$-sound above $\beta$ relative to $p$ iff $\mathcal{C}_0(\mathcal{M}) = \hull^\mathcal{M}_{n}(\beta \cup \{p\})$. We will not mention the parameter if $\mathcal{M}$ is $n$-sound above $\beta$ relative to $p_{n}(\mathcal{M})$. If $\mathcal{M}$ is $n$-sound above $\rho_{n}(\mathcal{M})$, we simply say that $\mathcal{M}$ is $n$-sound.

A potential premouse is then a premouse if all its initial segments are $n$-sound for all $n$. We can now also define fine structural ultrapowers. Let $\mathcal{M}$ be a premouse and let $F$ be an extender that measure all subsets of its critical point in $\mathcal{M}$. Let $n$ be such that $\crit(F) < \rho_n(\mathcal{M})$ and $\mathcal{M}$ is $n$-sound. Then $\ult_n(\mathcal{M},F)$ is the ultrapower formed using all equivalence classes $\left[f,a\right]_F$ where $a \in \finsubsets{\length(F)}$ and $f$ is a function with domain $\left[\crit(F)\right]^{\vert a\vert}$ that is $r\Sigma_n$-definable over $\mathcal{M}$ (in parameters).

\begin{lemma}\label{fscof}
   Let $\mathcal{M}$ be a premouse, and let $\kappa \in \mathcal{C}_0(\mathcal{M})$ be a regular cardinal there. Assume $\rho_{n+1}(\mathcal{M}) \leq \beta < \kappa \leq \rho_n(\mathcal{M})$ for some $n$ such that $\mathcal{M}$ is $(n + 1)$-sound above $\beta$. Then $\cof(\kappa) = \cof(\rho_n(\mathcal{M}))$.
\end{lemma}

\begin{proof}
 For $\xi < \rho_n(\mathcal{M})$ we let $\mathcal{N}_\xi$ be the structure $\mathcal{M}\vert\vert \xi$ with $\Th^\mathcal{M}_n(\xi,p_n(\mathcal{M}))$ as an additional predicate. Let then $\kappa_\xi$ be the supremum of ordinals less than $\kappa$ which are $\Sigma_1$-definable over $\mathcal{N}_\xi$ from $p_{n+1}(\mathcal{M})$ and ordinals less than $\beta$. As all objects involved are elements of $\mathcal{M}$, we must have $\kappa_\xi < \kappa$. On the other hand $\sup\limits_{\xi < \rho_n(\mathcal{M})} \kappa_\xi = \kappa$ as $\mathcal{M}$ was $(n + 1)$-sound above $\beta$.
\end{proof}

An additional fact that we will need is that if $\mathcal{M}$ is an active (potential) premouse, then $\cof(\on \cap \mathcal{M}) = \cof((\mu^{\mathcal{M},+})^\mathcal{M})$. See the last remark of Chapter 1 in \cite{FSIT}.

\subsubsection{Iterability}

A (normal, $\omega$-maximal) iteration tree on a premouse $\mathcal{M}$ is a tuple $\mathcal{T} : = \<\<\mathcal{M}^\mathcal{T}_\alpha: \alpha \leq \length(\mathcal{T})\>,\<E^\mathcal{T}_\alpha: \alpha < \length(\mathcal{T})\>,D^\mathcal{T},\<\iota^\mathcal{T}_{\alpha,\beta}: \alpha \leq_\mathcal{T} \beta \leq \length(\mathcal{T})\>\>$ where $\mathcal{M}^\mathcal{T}_\alpha$ is a premouse for all $\alpha \leq \length(\mathcal{T})$ ($\mathcal{M}^\mathcal{T}_0 = \mathcal{M}$); $E^\mathcal{T}_\alpha$ is an extender from the $\mathcal{M}^\mathcal{T}_\alpha$-sequence for all $\alpha < \length(\mathcal{T})$, $\alpha < \beta$ implies $\length(E^\mathcal{T}_\alpha) < \length(E^\mathcal{T}_\beta)$; $\iota^\mathcal{T}_{\alpha,\beta}:\mathcal{C}_0(\mathcal{M}^\mathcal{T}_\alpha) \rightarrow \mathcal{C}_0(\mathcal{M}^\mathcal{T}_\beta)$ is the (possibly) partial iteration map for all $\alpha \leq_\mathcal{T} \beta \leq \length(\mathcal{T})$, it is total iff $D^\mathcal{T} \cap \left(\alpha,\beta\right]_{\leq_\mathcal{T}} \neq \emptyset$; $\leq_\mathcal{T}$ is the tree order on

$\length(\mathcal{T})$ with root $0$, if $\gamma + 1 \leq \length(\mathcal{T})$, then the $\mathcal{T}$-predecessor is the least $\beta$ such that $\crit(E^\mathcal{T}_\gamma) < \gen(E^\mathcal{T}_\beta)$, in that case $(\mathcal{M}^\mathcal{T}_{\gamma + 1})^*$ is the segment of $\mathcal{M}^\mathcal{T}_\beta$ to which

$E^\mathcal{T}_\gamma$ is applied, if $\lambda \leq \length(\mathcal{T})$ is a limit, then $b^\mathcal{T}_\lambda := \left[0,\lambda\right)_{\leq_\mathcal{T}}$ is a cofinal branch whose intersection with $D^\mathcal{T}$ is finite,

$\mathcal{M}^\mathcal{T}_\lambda$ must be the direct limit of $\la\mathcal{M}^\mathcal{T}_\alpha,\iota^\mathcal{T}_{\alpha,\beta}: \alpha \leq_\mathcal{T} \beta \in b^\mathcal{T}_\lambda\ra$;

finally, $\gamma + 1 \in D^\mathcal{T}$ if and only if $(\mathcal{M}^\mathcal{T}_{\gamma + 1})^* \neq \mathcal{M}^\mathcal{T}_\beta$.

A $\gamma$ iteration strategy $\Sigma$ for a premouse $\mathcal{M}$ is a function such that $\Sigma(\mathcal{T})$ is a cofinal and wellfounded branch for every iteration tree on $\mathcal{T}$ of limit length $\kleiner\gamma$ and with the property that $\Sigma(\mathcal{T} \restr \alpha) = \left[0,\alpha\right)_{\leq_\mathcal{T}}$ for all limit $\alpha < \length(\mathcal{T})$. $\mathcal{M}$ is $\gamma$-iterable if there exists a $\gamma$-iteration strategy for $\mathcal{M}$. We will just say $\mathcal{M}$ is iterable if it is $\gamma$-iterabe for all ordinals $\gamma$.

Let $\mathcal{M}$ be a premouse, $n < \omega$ and let $p_{n+1}(\mathcal{M}) = \<\xi_0,\ldots,\xi_{k-1}\>$. The $(n+1)$-th solidity witness $w_{n+1}(\mathcal{M})$ is a tuple $\<t_0,\ldots,t_{k-1}\>$ where \[t_i = \Th^\mathcal{M}_{n + 1}(\xi_i,\<\xi_0,\ldots,\xi_{i-1}\>).\] We say $\mathcal{M}$ is $(n + 1)$-solid if $w_{n+1}(\mathcal{M}) \in \mathcal{C}_0(\mathcal{M})$.

A core result of \cite{FSIT} is that any reasonably iterable $n$-sound premouse is $(n + 1)$-solid. Mitchell-Steel also showed the following with similar methods, see the remark after Theorem 8.2. Note that the requirement for unique branches can be replaced by the weak Dodd-Jensen property from \cite{JohnHST}.

\begin{lemma}[Condensation Lemma]\label{condens}
   Let $\mathcal{M} := (\vert\mathcal{M}\vert; \in, \vec{E},F)$ be a $(n + 1)$-sound premouse such that every countable hull of $\mathcal{M}$ has a $(\omega_1 + 1)$-iteration strategy. Let $\mathcal{N}$ be a premouse such that there exist an $r\Sigma_{n + 1}$-elementary embedding $\pi:\mathcal{C}_0(\mathcal{N}) \rightarrow \mathcal{C}_0(\mathcal{M})$ with $\crit(\pi) \geq \rho_{n+1}(\mathcal{N})$. Then $\mathcal{N}$ is an initial segment of $\mathcal{M}$ or of $\ult(\mathcal{M},\vec{E}_{\crit(\pi)})$.
\end{lemma}

Both these results use the notion of a phalanx (although this notion was not yet fully developed by the time of \cite{FSIT}) of which we too will have need. A phalanx is a tuple $\<\<\mathcal{M}_i: i \leq \alpha\>,\<\kappa_i: i < \alpha\>\>$ where $\mathcal{M}_i$ agrees with $\mathcal{M}_j$ up to $(\kappa^+_i)^{\mathcal{M}_j}$ for all $i < j \leq \alpha$.

Phalanxes are a natural byproduct of iteration trees, i.e. if $\mathcal{T}$ is a normal iteration tree on some premouse, then $\<\<\mathcal{M}^\mathcal{T}_i: i \leq \length(\mathcal{T})\>,\<\nu(E^\mathcal{T}_i): i < \length(\mathcal{T})\>\>$ is a phalanx.

We can then also define iterability on phanlanxes as a natural extension of the structure of iteration trees. Given a phalanx $\<\<\mathcal{M}_i: i \leq \alpha\>,\<\kappa_i:i < \alpha\>\>$ and an extender $E$ we can extend the phalanx by applying $E$ to $\mathcal{M}_i$ where $i$ is minimal with $\crit(E) < \kappa_i$. (Note we have to require that the length of $E$ is above $\sup\limits_{i < \alpha} \kappa_i$ to maintain ``normality".)

A notion of iteration then follows naturally. The most critical difference here is that we have to keep track above which element of the phalanx any given model of the iteration tree lies. The art of phalanx iteration lies in arranging things such that the last model of a co-iteration lies above the ``right" model.

\section{Forcing the Approachable Bounded Subset Property}\label{sec:forcing}
Our forcing notations is mostly standard.
We use the Jerusalem forcing convention by which
``a condition $p$ extends (is more informative than) $q$" is denoted by $p \geq q$. 
In general, names for a set $x$ in a generic extension will be denoted by $\name{x}$. If $x$ is in the ground model then its canonical name is denoted by $\check{x}$. \\

We denote our initial ground model by $V'$, which we assume to satisfy the following assumptions:
there are two increasing sequences $\langle \kappa_n \mid n < \omega \rangle$, $\langle \lambda_n \mid n < \omega \rangle$  of regular cardinals, with $\lambda_n < \kappa_{n+1} < \lambda_{n+1}$ for all $n$, and that each $\lambda_n$ is measurable of Mitchell order $o(\lambda_n) = \kappa_n$. \\
    For each $n < \omega$, let $\la U_{\lambda_n,\alpha} \mid \alpha < \kappa_n \ra$ be a $\mo$-increasing of normal measures on $\lambda_n$. I.e., $U_{\lambda_n,\alpha}$ belongs to the ultrapower by $U_{\lambda_n,\beta}$, whenever $\alpha < \beta$. 
Denote $\lambda = \cup_n \lambda_n$.


In order to apply our main extender-based forcing notion, we first force with a preparatory forcing $\po'$ over $V'$ to transform the Mitchell-order increasing sequences $\la U_{\lambda_n,\alpha} \mid \alpha < \kappa_n \ra$ of normal measures, to Rudin-Keisler increasing sequences.
For this, we force with a Gitik-iteration $\po'$ (\cite{Gitik-Iter}) for changing the cofinality of measurable cardinals between the cardinals $(\kappa_n,\lambda_n)$ for all  $n < \omega$.
Let $G' \subseteq \po'$ be a generic filter over $V'$, and set 
$V = V[G']$. 
We list a number of facts concerning the extensions in $V$ of the measures $\la U_{\lambda_n,\alpha} \mid \alpha < \kappa_n \ra$ from $V'$. The analysis leading to these facts can be found in \cite{Gitik-Iter}, or \cite{HOD1} for a similar type of poset. 
The Mitchell-order increasing sequence $\la U_{\lambda_n,\alpha} \mid \alpha < \kappa_n \ra$ extends to a Rudin-Keisler increasing sequence of $\lambda_n$-complete measures $\la U^*_{\lambda_n,\alpha} \mid \alpha < \kappa_n\ra $, with Rudin-Keisler projections $\pi^{n}_{\beta,\alpha} : \lambda_n \to \lambda_n$ for each $\alpha < \beta < \kappa_n$. We note that the least measure $U^*_{\lambda_n,0}$ remains normal.  
We denote for each $n < \omega$ the linear directed system of measures $\{ U^*_{\lambda_n,\alpha}, \pi^n_{\beta,\alpha} \mid \alpha \leq \beta < \kappa_n\}$ by $E_n$, and 
further denote each $U^*_{\lambda_n,\alpha}$ by $E_n(\alpha)$. 
Let
\[j_{E_n} : V \to M_{E_n} = \ult(V,E_n) = \dirlim_{\alpha < \kappa_n}\ult(V,E_n(\alpha))
\]
Each measure $E_n(\alpha)$ can be derived from $j_{E_n}$ using a generator $\gamma^{E_n}_\alpha < j_{E_n}(\lambda_n)$.
The following list summarizes the key properties of the extenders $E_n$:
\begin{fact}\label{fact:E_n}
${}$
\begin{enumerate}
    
    \item $\cp(j_{E_n}) = \lambda_n$ and $M_{E_n}^{<\kappa_n} \subseteq M_{E_n}$
    \item $\gamma^{E_n}_0 = \lambda_n$ and $\la \gamma_\alpha \mid \alpha < \kappa_n\ra$ is a strictly increasing and continuous sequence
    \item $\gamma^{E_n} = \sup_{\alpha < \kappa_n}\gamma^{E_n}_\alpha$ is strongly inaccessible in $M_{E_n}$, and we may assume that there exists a function $g_n : \lambda_n \to \lambda_n$ such that $\gamma^{E_n} = j_{E_n}(g_n)(\lambda_n)$
    \item for each $\alpha < \beta < \kappa_n$, $E_n(\alpha)$ is strictly weaker than $E_n(\beta)$ in the Rudin-Keisler order. I.e., for every $A \in E_{n}(\alpha)$ there is $\nu \in A$ such that $\pi_{\beta,\alpha}^{-1}(\{\nu\})$ is unbounded in $\lambda_n$. 
    \item for every $\alpha < \kappa_n$ and $h : \lambda_n \to \lambda_n$ such that 
    $j_{E_n}(h)(\gamma^{E_n}_\alpha) < \gamma^{E_n}$.
    $j_{E_n}(h)(\gamma^{E_n}_\alpha) < \gamma^{E_n}_\beta$ for all $\beta > \alpha$.
\end{enumerate}
\end{fact}

Next, we force over $V$ with a short extender-based-type forcing $\po$, associated with the extenders $E_n$, $n < \omega$. 
$\po$ is a variant of the forcing in \cite{HOD2}
. Extending the arguments of \cite{HOD2}, we focus here on the generic scale associated with the extender-based-forcing, and use it to analyze the possible internally approachable structures in the generic extensions. This approach follows the one taken in \cite{BN-TS}, where an extender-based forcing has been used to obtain results concerning internally-approachable structures witnessing that ground model sequences $\la S_n \mid n < \omega\ra$ being tightly-stationary.

\begin{definition}\label{def:posetP}

Conditions $p \in \po$ are sequences $p = \la p_n \mid n < \omega\ra$ such that there is some $\ell<\omega$ for which the following requirements hold:
\begin{enumerate}
    
    \item for $n < \ell$,  
$p_n = \la f_n \ra$, where $f_n : \lambda^+ \to \lambda_n$ is a partial function of size $|f_n| \leq \lambda$, with $0 \in \dom(f_n)$ and both $f_n(0),g_n(f_n(0))  < \lambda_n$ are strongly inaccessible cardinals
\item For $n \geq \ell $, $p_n = \la f_n, a_n, A_n\ra$, where
$f_n$ is as above, $a_n : \lambda^+ \to \kappa_n$ is a partial continuous and order-preserving function, whose domain 
 is a closed and bounded set of $\lambda^+$ of has size $|a_n| < \kappa_n$. \\
We define $\mc(a_n)$ to be $a_n(\max(d_n)) = \max(\rng(a_n))$, and require that the set
$A_n$ to be contained in $\lambda_n \setminus \lambda_{n-1}$ and belong to $E_n(\mc(a_n))$.
\item $\dom(a_n) \cap \dom(f_n) = \emptyset$ and $\dom(a_n) \subseteq \dom(a_{n+1})$ for every $n \geq \ell$,
$a_n(0) = 0$, and for every $\delta \in \cup_n \dom(f_n)$ there exists some $m < \omega$ such that $\delta \in \dom(a_m)$. 
\end{enumerate}

For a condition $p \in \po$ as above, we denote $\ell,f_n,a_n,A_n$ by $\ell^p,f_n^p,a_n^p,A_n^p$ respectively. 
Direct extensions and end-extensions of conditions are defined as follows. 
A condition $p^*$ is a direct extension of $p$, if $\ell^{p^*} = \ell^p$, 
$f_n^p \subseteq f_n^{p^*}$ for all $n < \omega$, and $a_n^{p} \subseteq a_n^{p^*}$, 
$A_n^{p^*} \subseteq (\pi^n_{\mc(a_n^{p^*}),\mc(a_n^{p})})^{-1}A_n^p$ for all $n \geq \ell^p$.

For every $\nu \in A^p_n$, define $p_n \fr \la \nu\ra = \la f'_n\ra$, where 
\[
f'_n = f^{p_n} \cup \{ \la \alpha,\pi^n_{\mc(a_n^p),a_n(\alpha)}(\nu) \ra \mid \alpha \in \dom(a_n^p) \}.
\]

If $\vec{\nu} = \la \nu_{\ell^p},\dots,\nu_{n-1}\ra$ belong to $\prod_{i= \ell^p}^{n-1} A^p_{i}$, we define the end extension of $p$ by $\vec{\nu}$,  denoted $p \fr \vec{\nu}$, to be the condition $p' = \la p'_n \mid n < \omega\ra$, defined by $p'_k = p_k$ for every $k \not\in \{\ell^p,\dots, n-1 \}$, 
and $p'_k = p_k \fr \la \nu_k\ra$ otherwise. 
A condition $q \in \po$ extends $p$ if $q$ is obtained from $p$ by a finite sequence of end-extensions and direct extensions. \text{Equivalently}, $q$ is a direct extension of an end-extension $p \fr \vec{\nu}$ of $p$. 
Following the Jerusalem forcing convention,
we write $p \geq q$ if $p$ extends $q$, and $p \geq^* q$ if $p$ is a direct extension of $q$.
\end{definition}

\begin{notation}\label{notation:proj}
We introduce the following notational convention for the Rudin-Keisler projections $\pi^n_{\al,\beta}$ 
to be applied in the context 
of the forcing $\po$.
Let $p$ be a condition and $\nu \in A^p_n$ for some $n \geq \ell^p$ and $\al \in \dom(a_n^p)$. 
We write $\pi^p_{\mc(p),\al}(\nu)$ for $\pi^n_{\mc(a_n^p),a_n(\al)}(\nu)$.\\ \footnote{Note that the index $n$ is determined from the fact that $\nu \in A^p_n \subseteq \lambda_n \setminus \lambda_{n-1}$.}
Similarly, for a sequence 
$\vec{\nu} = \la \nu_{i}\ra_{\ell^p \leq i < n} \in \prod_{\ell^p \leq i < n} A^{p}_i$,
we write $\pi^p_{\mc(p),\al}(\vec{\nu})$ for the projected sequence
$\la \pi^p_{\mc(p),\al}(\nu_i) \mid \ell^p \leq i < n   \ra$.
\end{notation}

We proceed to list several standard basic properties of the poset $\po$, refering the reader to \cite{HOD2} for details.
\begin{lemma}\label{lemma-meetdenseset}
${}$
\begin{enumerate}
    \item $(\po,\leq,\leq^*)$ is a Prikry-type forcing
    \item for each $p \in \po$, the direct extension order $\leq^*$ of $\po/p$ is $\kappa_{\ell}$-closed
    \item $\po$ satisfies the $\lambda^{++}$.c.c
    \item(Strong Prikry Property) Let $D \subseteq \po$ be a dense open set. For every $p \in \po$ there are $p^* \geq^* p$ and $n < \omega$, such that for every 
$\vec{\nu} \in \prod_{\ell^p \leq i < n} A_{i}^{p^*}$, $p^* \fr \vec{\nu} \in D$.
\end{enumerate}
\end{lemma}

It is routine to verify that the above properties imply that $\po$ does not add new bounded subsets to $\lambda$, and does not collapse $\lambda^{++}$. We extend our analysis of $\po$ below to show that that it preserves $\lam^+$. This result can be also derived using a standard application of the Weak Covering Theorem. 
To extend our study of the poset $\po$, we introduce a notation of orderings $\leq^m$, $m < \omega$, which refine the direct extension ordering $\leq^*$.
\begin{notation}
Let $p,q$ be two conditions in $\po$. For $m < \omega$ we write $p \leq^m q$ if $p \leq^* q$ and 
$a^p_n = a^q_n$, $A^p_n = A^q_n$ for all $n <m$.
\end{notation}
Therefore, for each $m < \omega$, $\leq^m$ is $\kappa_m$-closed and $\leq^{m+1}\thinspace \subseteq \thinspace\leq^m$.

\begin{lemma}\label{lemma-densesets}
   Let $\theta > \lambda^+$ regular, $\mo$ be a well-ordering of $H_\theta$, and $M \elem (H_\theta;\in,\mo)$ satisfying $\po \in M and |M| = \lambda, V_\lambda\subseteq M$.
   Suppose that there exists an enumeration $\vec{D} = \la D_\mu \mid \mu < \lambda\ra $ of all dense open subsets of $\po$ in $M$,
   so that $\vec{D} \uhr \nu \in M$ for every $\nu < \lambda$.
   Then for every condition $p \in \po\cap M$ and $\ell^*$, $\ell^p \leq \ell^* < \omega$, there exists $p^* \geq^{\ell^*} p$ 
   so that for each dense open set $D \in M$ there are - 
   \begin{itemize}
   \item $q \in M$ with $p \leq^{\ell^*}q \leq^{\ell^*} p^*$, 
   \item a finite ordinal $n^D < \omega$, and
   \item  a function $N^D : \prod_{\ell^p\leq i < n^D}A^q_i \to \omega$,
   \end{itemize}
   such that for every pair of sequences $\vec{\nu}^1,\vec{\nu}^2$, satisfying 
   \[\vec{\nu}^1 \in \prod_{i = \ell^p}^{n^D-1}A^q_i, \text{ and } 
   \vec{\nu}^2 \in \prod_{i = n^D}^{N^D(\vec{\nu}^1)}A^q_i,\]
   the condition $q \fr \vec{\nu}^1 \fr \vec{\nu}^2$ belongs to $D$. 
\end{lemma}

\begin{remark}\label{rmk:properness}
We note that the condition $q \fr \vec{\nu}^1 \fr \vec{\nu}^2$ in the statement of the Lemma belongs 
to $M$ as $V_\lambda \subseteq M$. 
Therefore, Lemma \ref{lemma-densesets} implies that $p^*$ is a generic condition for $(M,\po)$, namely, it forces the statement $ \name{G} \cap \check{M} \cap \check{D} \neq \emptyset$ for every dense open set $D \in M$. 
\end{remark}

\begin{proof}
We assume for notational simplicity that $\ell^p = 0$. The proof for the general case is similar. 
We fix for each $n < \omega$ a bijection $\psi_n : \lambda_n \to [\lambda_n]^{n+1} \times \lambda_n$ in $M$.
Our final condition $p^*$ will be obtained as a limit of a carefully constructed sequence $\la p^n \mid \ell^* \leq n < \omega\ra$, starting from $p^{\ell^*} = p$, and consisting of conditions in $M$.
 Moreover, it will satisfy $p^n \leq^{n+1} p^{n+1}$ for all $n \geq \ell^*$.
Suppose that $p^n$ has been defined for some $n \geq \ell^*$. 
Our goal is to construct an extension $p^{n+1} \geq^{n+1} p^n$, 
so that for every ordinal $\mu$, $\lambda_{n-1} \leq \mu < \lambda_{n}$ there exists a function 
    $N^{D_\mu} : \prod_{i\leq n}A^{p^{n+1}}_i \to \omega $ so that for every 
    \[
    \vec{\nu}^1 \in \prod_{i\leq n}A^{p^{n+1}}_i \text{ and  } \thinspace
    \vec{\nu}^2 \in \prod_{n+1 \leq  i < N^{D_\mu}(\vec{\nu}^1)}A^{p^{n+1}}_i
    \]
    $p^{n+1}\fr \vec{\nu}^1 \fr \vec{\nu}^2 \in D_\mu$. We note that this will guarantee $n^{D_\mu}= n+1$ for 
    $\lambda_n \leq \mu < \lam_{n+1}$.
$p^{n+1}$ will be constructed from $p^n$ in $(\lambda_n+1)$-many steps, using two sequences of condition parts, 
$ \la \vec{f}^i \mid i \leq \lambda_n\ra$ and $\la q^i \mid i \leq \lambda_n  \ra$ which satisfy the following requirements:
\begin{enumerate}

    \item 
$\vec{f}_i = \la f_{i,0}, f_{i,1}, \dots, f_{i,n}\ra \in M$ is an $(n+1)$-tuple, consisting of Cohen functions, $f_{i,k} : \lambda^+ \to \lambda_k$ of size at most $\lambda$. 
For $i = 0$, $\vec{f}_0 = \la f^{p^n}_0,\dots,f^{p^n}_n\ra$ is the tuple of the Cohen functions of the first $(n+1)$ Cohen components of $p^{n+1}$.
\item For each $k\leq n$, the sequence $f_{i,k},i\leq \lambda_n$ is increasing in $\subseteq$, 
and $\dom(f_{i,k}) \cap \dom(a_k^{p^n}) = \emptyset$ for all $i \leq \lambda_n$.
\item 
$q^i = \la q^i_m \mid n < m <  \omega\ra$ consists of tail segments of conditions in $\po$ starting from the $(n+1)$-th component,
and $q^0 = p^n\setminus n+1 = \la p^n_m \mid n < m< \omega\ra$.
\item the sequence $q^i$,$ i \leq \lambda_n$ will be $\leq^*$-increasing in the obvious sense. 
\end{enumerate}
The construction of the two sequences will be internal to $M$, and definable from $p^n,\vec{D}\uhr \lambda_{n+1}$, and using 
the fixed well-ordering $\mo$ of $H_\theta$.
Let $\delta \leq \lambda_n$ and suppose that $\la q^i,\vec{f}_i \mid i < \delta\ra$ has been defined and belongs to $M$. 
If $\delta$ is a limit ordinal, we define $\vec{f}_\delta =\la f_{\delta,k} \mid k \leq n\ra$ by $f_{\delta,k} = \bigcup_{i<\delta}f_{i,k}$. 
Similarly,  $q_\delta$ is taken to be the supremum in the direct extension ordering of $q^i$, $i < \delta$, 
which is possible due to the fact that for each
$k > n$, the direct extension ordering of the $k$-th components $q^i_k$, is $\kappa_{k+1}$-closed, and $\kappa_{k+1} > \lambda_n$. Therefore, for every $k > n$, we define  $q^\delta_k = (f_k^{q^\delta}, a_k^{q^\delta}, A_k^{q^\delta})$ where
\[f_k^{q^\delta} = \bigcup_{i<\delta}f_k^{q^i}, a_k^{q^\delta} = \bigcup_{i<\delta}a_k^{q^i} \cup \{ (\alpha,\gamma)\}, \text{ and }
A_k^{q^\delta} = \bigcap_{i < \delta} (\pi^n_{\gamma,\mc(a_k^{q^i})})^{-1}A_k^{q^i}, 
\]
where $\alpha = \sup\left( \bigcup_{i<\delta} \dom(a_k^{q^i})\right)$, and
$\gamma = \sup\left( \bigcup_{i<\delta} \rng(a_k^{q^i})\right)$.
Clearly, $q^\delta \in M$. Suppose now that $\delta = i+1$ is a successor ordinal. We appeal to our fixed bijection 
 $\psi_n : \lambda_n \to [\lambda_n]^{n+1} \times \lambda_n$, and consider $ \psi_n(i) = (\vec{\nu}^i,\mu^i)$, where $\vec{\nu}^i \in [\lambda_n]^{n+1}$ and $\mu^i < \lambda_n$. 
 We proceed as follows:
 If $\vec{\nu}^i \not\in \prod_{i \leq n} A^{p^n}_i$ we make no change,
 setting $\vec{f}_\delta = \vec{f}_i$ and $q^\delta = q^i$.
 Otherwise, $\vec{\nu}^i \in \prod_{i \leq n} A^{p^n}_i$ and we consider the associated functions
 $\la g_{i,0},\dots, g_{i,n} \ra$, defined by 
 \[g_{i,k} = \{ \la \alpha,\pi^k_{\mc(a_k^{p^n}),a^{p^n}_k(\alpha)}(\nu) \ra \mid \alpha \in \dom(a_k^{p^n}) \}.\]
 
We note that since $\dom(g_{i,k}) = \dom(a^{p^n}_{k})$, it is disjoint from $\dom(f_{i,k})$, 
and we can therefore take their unions  $f^*_{i,k} = f_{i,k} \cup g_{i,k}$, $k < n$ to define a sequence of functions
 $\vec{f_i^*} = \la f^*_{i,0},\dots, f^*_{i,n}\ra$. 
 By concatenating the $(n+1)$-sequence $\vec{f_i^*}$ with the tail $q^i$, we get a condition  
 $q^i_* = \vec{f_i^*} \fr q_i \in \po$ with $\ell^{q^i_*} = n+1$, to which we apply the last clause of Lemma 
 \ref{lemma-meetdenseset} (Strong Prikry Property) and find a direct extension $q^i_{**} \geq^* q^i_*$ and an integer $N$ so that for every $\vec{\nu} \in \prod_{k \leq N} A^{q^i_{**}}_{n+1 + k}$, 
 $q^i_{**} \fr \vec{\nu}$ belongs to $D_{\mu_i}$. 
 Specifically, we choose $q^i_{**} \in M$ to be such a condition which is minimal according to the fixed well-ordering $\mo$ of $H_\theta$, 
 and define 
 \begin{itemize}
 \item $N^{D_{\mu_i}}(\vec{\nu}^i) =  N$,  
 \item $\vec{f}_{\delta} = \la f_{\delta,k} \mid k \leq n\ra$ with 
 $f_{\delta,k} = f^{q^i_{**}}_k \setminus g_{i,k}$,\footnote{thus, $\dom(f_{\delta,k})$ is disjoint from $\dom(g_{i,k}) = \dom(a^{p^n}_k)$.}
 and 
 \item  $q^\delta = \la q^\delta_m \mid m \geq n+1\ra$ with $q^\delta_m = (q^i_{**})_m$ for every $m \geq n+1$. 
 \end{itemize}

 Finally, given $\vec{f}_{\lambda_n} = \la f_{\lambda_n,k} \mid k \leq n \ra$ and  $q^{\lambda_n} = \la q^{\delta_n}_m \mid m \geq n+1 \ra$. 
 we define $p^{n+1} \geq^* p^n$ by setting $a_m^{p^{n+1}} = a_m^{p^{n}}$ and $A_m^{p^{n+1}}  = A_m^{p^{n}} $ 
 and $f_m^{p^{n+1}} = f_{\lambda_n,m}$ for $m \leq n$, and  $p^{n+1}_m = q^{\lambda_n}_m$ for $m \geq n+1$.
 Our use of the well-ordering $\mo$ throughout the construction guarantees that $p^{n+1} \in M$.
 
 This concludes the construction of the sequence $\la p^n \mid \ell^* \leq n <\omega\ra$. 
 We now define $p^* \geq^* p$ by $p^*_n = p_n^{n}$. It is straightforward to verify from the construction that $p^* \geq \ell^*$ satisfies the conclusion in the statement of the Lemma. 
\end{proof}

We now show that there are plenty of models $M$, satisfying the conclusion of Lemma \ref{lemma-densesets}.

\begin{proposition}\label{prop:properness}
Let $\vec{M} = \la M_\al \mid \al < \lambda^+\ra$ be an internally approachable sequence (i.e., $\vec{M}\uhr\beta \in M_{\beta+1}$ 
for every $\beta < \lambda^+$) $\subseteq$-increasing and continuous sequence of elementary substructures $M_\alpha \elem (H_\theta;\in,\mo)$ of size $|M_\alpha| = \lambda$, and satisfy $M_\alpha \cap \lambda^+ \in \lambda^+$.
For every limit ordinal $\al < \lam^+$ of $\cf(\al) = \omega$,  satisfying $\alpha = M_\alpha \cap \lambda^+$, $\ell^* < \omega$, and $p \in M_\alpha$, there exists a direct extension $p^* \geq^{\ell^*} p$ satisfying the conclusion of Lemma \ref{lemma-densesets} with respect to $M= M_\al$.\\
Moreover, if the approachable ideal on $\lam^+$ is trivial, i.e., $I[\lam^+] = \lam^+$, then the requirement of $\cf(\al) = \omega$ can be removed. 
\end{proposition}

\begin{proof}
 Suppose first that $\cf(\al) = \omega$ and let $\la \alpha_n \mid n < \omega\ra$ be a cofinal sequence in $\alpha$.
 Then $M_\alpha = \cup_n M_{\alpha_n}$, and for each $n < \omega$ since $M_{\alpha_n} \in M_\alpha$, there exists an enumeration $\vec{D}^n = \la D^n_\mu \mid \mu < \lambda\ra \in M_\alpha$ of all dense open subsets of $\po$ in  $M_{\alpha_n}$.
Using bijections from $\lambda_n \times n$ to $\lambda_n$, we can form a sequence 
$\vec{D} = \la D_\mu \mid \mu < \lambda\ra$ so that for every $n < \omega$, $\vec{D}\uhr\lambda_n$ enumerates $\vec{D}^i\uhr\lambda_n$ for each $i < n$. Therefore $\vec{D}$ enumerates all dense open sets of $\po$ in $M_\alpha$ and satisfies $\vec{D}\uhr\beta \in M_\alpha$ for every $\beta < \lambda$. It follows from Lemma \ref{lemma-densesets} that for every condition $p \in M_\alpha$ and $\ell^* <\omega$ there exists a direct extension $p^* \geq^{\ell^*} p$ as in the statement of the lemma. This concludes the first part of the statement. \\

Suppose now that $I[\lam^+] = \lam^+$. We proceed to prove by induction on limit ordinals $\al < \lam^+$ with $\alpha = M_\alpha \cap \lambda^+$, that for every $\ell^* < \omega$ and $p \in M_\alpha$, there is $p^* \geq^{\ell^*} p$ satisfying the desirable property for $M_\al$.
Let $\al$ be such an ordinal and assume the statment holds for all $\beta < \al$. If $\cf(\al) = \omega$ we are done by the first case above. 
Therefore, suppose that $\cf(\alpha) = \rho$ is an uncountable regular cardinal. 
Since $I[\lam^+] = \lam^+$
, there exists a closed and unbounded subset $X \subset \alpha$ of order-type $\otp(X) = \rho$ so that $X \cap \beta \in M_{\alpha'}$ whenever $\beta < \alpha'$ , $\alpha' \in \{\alpha\} \cup X$. Moreover, since $\vec{M}\uhr\beta$ belongs to $M_{\alpha'}$, so does
$\vec{M}\uhr(X \cap \beta) = \la M_\gamma \mid \gamma \in X \cap \beta\ra$.
Given $\ell^* < \omega$ as in the statement of the claim, we further increase it to assume that $\kappa_{\ell^*} > \rho$. Let $\la \beta_i \mid i \leq \rho\ra$ be an increasing enumeration 
of the limit points $\beta$ in $X \cup \{\alpha\}$
 which satisfy that $M_\beta \cap \lam^+ = \beta$.
Given $p \in M_\alpha$, we may assume that $p \in M_{\beta_0}$ and denote it by $p^0$. Then, by applying the inductive assumption and using the well ordering $\mo$, we form a sequence of conditions
$\la p^{i} \mid i \leq \rho\ra$ which is increasing in $\leq^{\ell^*}$, so that for each $i < \rho$
$p^i \in M_{\beta_{i+1}}$ and $p^{i+1} \geq^{\ell^*} p^i$ is the $\mo$-minimal such extension, which is satisfies the conclusion of Lemma \ref{lemma-densesets} for $M_{\beta_{i+1}}$.
Suppose now that $j \leq \rho$ is limit. Then every initial segment of $\la p^i \mid i < j\ra$ belongs to $M_{\beta_j}$, and the sequence has an upper bound in $\leq^{\ell^*}$ since this ordering is $\kappa_{\ell^*}$-closed and $\kappa_{\ell^*} > \rho$. Defining the upper bound by $p^j$, it follows from the continuity of the sequence $\vec{M}$ that $p^j$ satisfies the desirable property for $M_{\beta_{j}}$.
In particular, for $j = \rho$, we obtain a suitable condition $p^* = p^\rho$ for $M = M_\al$.
\end{proof}

The following consequences of Lemma \ref{lemma-densesets} and Proposition
\ref{prop:properness} will play a key role in our arguments concerning Approachable Bounded Subset Property in $V[G]$.


\begin{lemma}\label{lemma-functionname}
 Let $\name{F}$ be a $\po$-name of a function from $\lambda^{<\omega}$ to ordinals, and $p \in \po$.
 There is a direct extension $p^* \geq^* p$ and a function $f^* : [\lambda]^{<\omega} \times [\lambda]^{<\omega} \to \On$ which provide the following recipe for 
 deciding the $\po$-names of ordinals $\name{F}(\vec{\mu})$, $\vec{\mu} \in [\lambda]^{<\omega}$:
 
For every $\vec{\mu} \in [\lambda]^{<\omega}$ there are $n^{\vec{\mu}} <\omega$ and a function  
$N^{\vec{\mu}} : [\lambda]^{<\omega} \to \omega$ such that
for every $\vec{\nu}^1\in \prod_{\ell^p \leq i < n^{\vec{\mu}}} A_i^{p^*}$ and
 $\vec{\nu}^2 \in \prod_{n^\mu \leq i<N^{\vec{\mu}}(\vec{\nu}^1)} A_i^{p^*}$, 
\[
p^* \fr \vec{\nu}^1 \fr \vec{\nu}^2 \Vdash \name{F}(\vec{\mu}) = \can{f^*}(\can{\vec{\mu}},\can{\vec{\nu}}^1 \fr \can{\vec{\nu}}^2).
\]
\end{lemma}

\begin{proof}
Let $M \elem (H_\theta;\in,\mo)$ be a model of size which satisfies the assumption of Lemma \ref{lemma-densesets} and has $\name{F},p \in M$ (the proof of Proposition \ref{prop:properness} shows that such structures exist). Since $\lambda \subseteq M$ then for every
$\vec{\mu} \in [\lambda]^{<\omega} $, the dense open set
\[E_{\vec{\mu}}= \{ q \in \po \mid \exists \xi \in \On, q \Vdash \name{F}(\can{\vec{\mu}}) = \can{\xi}\}\]
belongs to $M$. 
By taking $p^* \geq^* p$ as in the statement of Lemma \ref{lemma-densesets} we obtain the desired extension of $p$.
\end{proof}

\begin{corollary}\label{cor:lam+Preserved}
   $\po$ preserves $\lambda^+$.
\end{corollary}
\begin{proof}
   If $\name{F}: \lambda \to \lambda^+$ is a $\po$-name of a function, then by Lemma \ref{lemma-functionname} for every condition $p$ there are $p^* \geq^* p$ and a function $f^* : [\lambda]^{<\omega} \times [\lambda]^{<\omega} \to \lambda^+$ in $V$, so that $p^*$ forces $\rng(\name{F})$ is contained in $\rng(f^*)$.
\end{proof}

Let $G \subseteq \po$ be a generic filter. By a standard density argument, for every $\alpha < \lambda^+$ and $n < \omega$ there exists $p \in G$ so that $\ell^p > n$ and  $\alpha \in \dom(f^p_n)$.
We define the generic scale $\la t_\alpha \mid \alpha < \lambda^+\ra$ by $t_\alpha(n)  = f^p_n(\alpha)$ for any such a condition $p \in G$.

Recalling that our setup includes that $a_n(0) = 0$ and $E_n(0)$ is a normal measure on $\lambda_n$, we get that the sequence $\la \rho_n \mid n < \omega\ra$, given by $\rho_n = t_0(n)$, is generic over $V$ for the diagonal Prikry forcing with the sequence of normal measures $\la E_n(0) \mid n  <\omega \ra$.

Recall that for every $n < \omega$, 
there exists a function  $g_n : \lam_n \to \lam_n$ so that  $j_{E_n}(g_n)(\lambda_n)$ is the supremum of the generators of $E_n$, and is inaccessible in $M_{E_n}$.
It follows from a standard density argument that the sequence $\la t_\alpha \mid \alpha < \lambda^+\ra$ is a scale in the product $\prod_{n} g_n(\rho_n)$, and that $g_n(\rho_n) < \lambda_n$ is regular for almost all $n < \omega$. Moreover, it is straightforward to verify that our assumption that the functions $a_n$ in conditions $p \in \po$ are continuous and have closed domains, implies that the scale $\la t_\alpha \mid \alpha < \lambda^+\ra$ is continuous. 

\begin{notation}
In $V[G]$, we denote $g_n(\rho_n)$ by $\tau_n$.
\end{notation}

\begin{theorem}\label{thm:ABSPnew}
The Approachable Bounded Subset Property (ABSP) holds in $V[G]$ with respect to the sequence 
$\la \tau_n \mid n < \omega\ra$.
\end{theorem}

\begin{proof}
Suppose otherwise, then there exists a stationary set $\S \subseteq \power_\lambda(H_\theta)$ of 
internally approachable structures $N \elem (H_\theta;\in)$ 
  such that for every $N \in \S$ and $n < \omega$ there is a function 
 $F^N_n : [\lambda]^{k^N_n} \to \lambda$ in $N$, of a finite arity $k^N_n < \omega$, and a finite sequence of distinct numbers  $\vec{d}^{N,n} = \la d^{N,n}_{0},\dots,d^{N,n}_{k^N_n}\ra \subseteq \omega\setminus n$, satisfying 
\[ \chi_N(\tau_{d^{N,n}_0}) \leq F^N_n
\left(\chi_N(\tau_{d^{N,n}_1} ),\dots,\chi_N(\tau_{d^{N,n}_{k^N_n}})\right) < \tau_{d^{N,n}_0}.
\]
By Lemma \ref{lem:IApressingdown}, applied to the assignments
$N \mapsto \la F^N_n\mid n < \omega\ra$ and $N \mapsto \la \vec{d}^{N,n} \mid n < \omega\ra$, 
there exists a stationary set $S^* \subseteq \lam^+$ and two fixed sequences 
$\la F_n \mid n < \omega\ra$, $\la \vec{d}^n \mid n < \omega\ra$, 
with $F_n : [\lambda]^{k_n} \to \lambda$ and 
$\vec{d}^n = \la d^n_0,\dots d^n_{k_n}\ra$, such that 
for every $\delta \in S^*$ there exists $N \in \S$ so that
$\delta = \chi_N(\lam^+)$,  $\la F_n^N \mid n < \omega\ra = \la F_n \mid n < \omega\ra$, and
$\la \vec{d}^{N,n} \mid n < \omega\ra = \la \vec{d}^n \mid n < \omega\ra$.
For each $\delta \in S^*$ there are  $m_\delta < \omega$ and $N \in \S^*$ such that for every $n \geq m_\delta$, 
$t_\delta(n) = \chi_N(\tau_n)$ and thus, 
\begin{equation}\label{equation:crux}
   t_\delta(d^{n}_0) \leq F_n
\left(t_\delta({d^{n}_1} ),\dots,t_\delta({d^{n}_{k_n}})\right) < \tau_{d^{n}_0} 
\end{equation}

We move back to $V$ to contradict the above, and complete the proof. 
Let $p$ be a condition forcing the statement of (\ref{equation:crux}) with respect to the $\po$-names $\name{S^*}$, $\la \name{F_n} \mid n < \omega\ra$, and $
\la {\name{\vec{d}}^n} \mid n < \omega\ra$.
By taking a direct extension if needed, we may assume $p$ decides the integer values for $\vec{d}^n \subseteq \omega\setminus n$, for all $n < \omega$. 
Apply Lemma \ref{lemma-functionname} repeatedly for each $F_n$, $n < \omega$, to form sequences, 
$\la p^n \mid n <\omega\ra$ of $\leq^*$-extensions of $p$, and $\la f^n \mid n < \omega\ra$ of functions, $f^n : [\lambda]^{<\omega} \times [\lambda]^{<\omega} \to \lambda$, so that for each $n < \omega$, $p^n \geq^* p^{n-1}$ and $f^n$ are formed to satisfy the conclusion of Lemma \ref{lemma-functionname} with respect to $\name{F_n}$. 
We define \[\alpha = \sup\left(\bigcup_{n,m}(\dom(a^{p^n}_m) \cup \dom(f^{p^n}_m))\right) + 1\]
and let $p^*$ be a common direct extension of $\la p^n \mid n < \omega\ra$ with
$\alpha = \max(\dom(a_m^{p^*}))$ for all $m \geq \ell^{p^*}$. 
Next, let $q$ be an extension of $p^*$ which forces $\can{\delta} \in \name{S^*}$ for some ordinal $\delta > \alpha$.
Since $q$ extends $p^*$, it is a direct extension of $p^* \fr \vec{\nu}^*$ for some 
$\vec{\nu}^* \in \prod_{\ell^{p^*} \leq i < \ell^*}A^{p^*}_i$. By taking a direct extension of $q$ if needed, we may also assume that $q$ decides the integer values $m_\delta$ from above and that $\delta \in \dom(a^q_n)$ for some $n < \omega$.

Next, we pick $n <\omega$ satisfying $n \geq m_\delta,\ell^q$ and $\delta \in \dom(a_n^q)$, and denote for ease of notation, 
$k_n$, $\la d^n_0,\dots,d^n_{k_n}\ra$ by $k$, $\la d_0,\dots,d_k\ra$ respectively. 
Our choice of $p^* \geq^* p^n$, and function $f^n$ guarantee that for every 
$\vec{\mu} = \la \mu_{d_1},\dots,\mu_{d_k} \ra \in [\lambda]^k$ there are $n_*^{\vec{\mu}} < \omega$ and a function 
\[N_*^{\vec{\mu}} : \prod_{\ell^q \leq i < n_*^{\vec{\mu}}}A^q_i \to \omega,\footnote
{
The vales $n_*^{\vec{\mu}}$ and $N^{\vec{\mu}}_*$ are the obvious shifts of $n^{\vec{\mu}}$ and $N^{\vec{\mu}}$ 
associated to $p^n,f^n$ from Lemma \ref{lemma-functionname}, resulting from the fact that the values given by 
 $n^{\vec{\mu}}$ and $N^{\vec{\mu}}$ apply to $p^n$ and hence to $p^*$, while
$q \geq^* p^* \fr \vec{\nu}^*$, already determined an initial segment $\vec{\nu}^*$ of a possible end-extension of $p^*$ which decides $\name{F_n}(\vec{\mu})$ according to the recipe of the lemma.
}\]
such that for every 
\[\vec{\nu}^1 \in \prod_{\ell^q \leq i < n_*^{\vec{\mu}}}A^q_i \text{ and }
\vec{\nu}^2 \in \prod_{ n_*^{\vec{\mu}} \leq i < N_*^{\vec{\mu}}(\vec{\nu}^1)}A^q_i,\]

denoting $ \vec{\nu}^* \fr \pi^{q}_{\mc(q),\al}(\vec{\nu}^1 \fr \vec{\nu}^2)$ by $\vec{\nu}$, we have that $p^* \fr \vec{\nu}$ forces $\name{F_n}(\can{\vec{\mu}}) = \can{f}^n(\vec{\mu},\vec{\nu})$.

Recalling that $q \geq^* p^* \fr \vec{\nu}^*$, we get that
in particular, $q \fr \vec{\nu}^1 \fr \vec{\nu}^2$, which extends $p^* \fr \vec{\nu}$ forces the same value, which depends only on $\pi^{q}_{\mc(q),\al}(\vec{\nu}^1 \fr \vec{\nu}^2)$.

To complete the argument, we will make use of the last fact, and the fact that $E_n(a^q_n(\delta))$ is strictly stronger than $E_n(a^q_n(\al))$ in the Rudin-Keisler ordering, to find many distinct choices of sequences $\vec{\nu}^1 \fr \vec{\nu}^2$, whose projections $\pi^p_{\mc(q),\delta}(\vec{\nu}^1 \fr \vec{\nu}^2)$ are fixed, as well as the values they force for 
$t_\delta(d_i)$, $1 \leq i \leq k$, yet they force many distinct values for $t_\delta(d_0)$. This will be used to find a condition which extends $p$ but forces $(\ref{equation:crux})$ to fail.

To this end, we fix first a sequence of $\la d_1,\dots,d_k\ra$-indices,
\[\vec{\nu}_{+} = \la \nu_{d_1},\dots,\nu_{d_k}\ra \in \prod_{i \in \{ d_1,\dots,d_k\}}A^q_i\] 
(note that we omit choosing a $d_0$-coordinate) and define $\vec{\mu} = \pi^q_{\mc(q),\delta}(\vec{\nu}_+)$. 

Let $j \leq k+1$ largest so that $d_i < n_*^{\vec{\mu}}$ for every $i < j$ and define 
$\vec{\nu}^1_+ = \la \nu_{d_0},\dots,\nu_{d_{j-1}}\ra$.
Therefore, in order to extend $\vec{\nu}^1_+$ to a relevant sequence $\vec{\nu}^1 \in \prod_{\ell^q \leq i < n_*^{\vec{\mu}}}A^q_i$, one needs to choose remaining coordinates 
\[\vec{\nu}^1_{-} \in \prod_{i \in [\ell^q, n_*^{\vec{\mu}}) \setminus \vec{d}}A^q_i.\] 
In particular, to every such sequence 
$\vec{\nu}^1_{-}$ we can assign the integer 
$N^{\vec{\mu}}_*(\vec{\nu}^1)$, where $\vec{\nu}^1 = \vec{\nu}^1_{-} \cup \vec{\nu}^1_+$, 
and by taking a direct extension of $q$ if needed, we may assume that the numbers 
$N_*^{\vec{\mu}}(\vec{\nu}^1)$ take a constant value $N$ 
for all $\vec{\nu}^1_{-} \in \prod_{i \in [\ell^q, n_*^{\vec{\mu}}) \setminus \vec{d}}A^q_i$.
Moreover, by increasing $N$ if necessary, we may also assume that $N > d_k$.  

With $N$ being fixed, we conclude that every choice of a sequence 
\[\vec{\nu}^2_{-}\in \prod_{i \in [n_*^{\vec{\mu}},N) \setminus \vec{d}}A^q_i\] 
will allow us to extend the remaining portion of the fixed sequence 
$\vec{\nu}_+$ to $\vec{\nu}^2 \in  \prod_{i \in [n_*^{\vec{\mu}},N)}A^q_i$, which together with a choice of $\vec{\nu}^1$ will produce a suitable extension $q \fr \vec{\nu}^1 \fr \vec{\nu^2}$ forcing 
\[\name{F_n}(\vec{\mu}) = f^n\left(\vec{\mu},\vec{\nu}^*\fr \pi^{q}_{\mc(q),\al}(\vec{\nu}^1 \fr \vec{\nu}^2) \right).\]
Following this recipe, we extend our fixed choice of $\vec{\nu}_+$ to a choice of all relevant coordinates for $\vec{\nu}^1 \fr \vec{\nu}^2$,  except for the coordinate of $i = d_0$. 
Namely, we extend $\vec{\nu}_+$ to a fixed sequence $\vec{\nu}_{++} \in \prod_{i \in [\ell^q,n_*^{\vec{\mu}}+N)\setminus \{d_0\}} A^q_i$. With the fixed choice $\vec{\nu}_{++}$, we derive a 
function $h : A^q_{d_0} \to \tau_{d_0}$, defined by
$h(\nu) = f^n(\vec{\mu},\pi^{q}_{\mc(q),\alpha}(\vec{\nu}_{++} \cup \{\nu\}))$ if the last ordinal value is below $\tau_{d_0}$, and $h(\nu) = 0 $ otherwise.
The properties of the function $f^n$ guarantee that $h(\nu)$ depends only on $\pi_{\mc(q),\alpha}(\nu)$, 
and by the last item on \ref{fact:E_n}, there is a subset $A^* \in E_{a_n(d_0)}$,  $A^* \subseteq A^{q}_{d_0}$, so that 
 $h(\nu) < \pi^q_{\mc(q),\delta}(\nu)$ for all $\nu \in A^*$. 
 Picking such an ordinal $\nu$, and setting $\vec{\nu}^1 \fr \vec{\nu}^2  = \vec{\nu}_{++} \cup \{\nu\}$, 
 we conclude that $p^* \fr \vec{\nu}^1 \fr \vec{\nu}^2 \geq p$ must force 
 \[ 
 h(\nu) = \name{F_n}\left(t_\delta(d_1),\dots, t_\delta(d_k)\right) < t_\delta(d_0).
 \]
 Contradicting the statement \ref{equation:crux} forced by $p$.
\end{proof}

\begin{corollary}
  ABSP holds in $V[G]$ with respect to $\la \tau_n \mid n < \omega\ra$ and thus, by Lemma \ref{lem:ImplicationsBetweenTLSandAFSP}, AFSP holds and there are no continuous scales on $\prod_n\tau_n$ which are essentially tree-like. 
\end{corollary}

\subsection{Down to $\aleph_\omega$}\label{sec:DownToAlephOmega}
We define a variant $\hat{\po}$ of the forcing $\po$ from the previous section, to obtain the result of Theorem \ref{thm:ABSPnew} in a model where $\la \tau_n \mid n < \omega\ra$ form a subsequence of the first uncountable cardinals.

Conditions $q \in \hat{\po}$ are pairs $q = \la p,h\ra$ of sequences, $p = \la p_n \mid n < \omega \ra$ and $h = \la h_{-1}^{\high} \ra \fr \la h_n \mid  n \in \omega\ra$ satisfying the following conditions:

\begin{enumerate}
    \item $p \in \po$, i.e., $p$ satisfies Definition 
    \ref{def:posetP} above
    
    \item for every $n < \ell^p$, $h_n = \la h_n^{\low},h_n^{\high}\ra$,
    is a pair of functions, which satisfy the following properties:
    \begin{itemize}
        \item $h_n^{\low}\in
      \coll( \rho_n^p,<\tau_{n}^p)$
     where 
    $\rho_n^p = f^p_n(0)$, 
     and $\tau_n^p = g_n(\rho_n^p)$,\footnote{By Definition \ref{def:posetP} $\rho_n^p < \tau_n^p< \lambda_n$ are both inaccessible for $n \geq 0$.} 
    \item 
    $h_n^{\high}\in \coll( (\tau_n^p)^+,<\rho_{n+1}^p)$ if $n < \ell^p-1$, and $h_{\ell^p-1}^{\high} \in \coll( (\tau_{\ell^p-1}^p)^+,<\lambda_{\ell^p})$. 
    \end{itemize}

    \item for every $n \geq \ell^p$, 
    $h_n = \la h_n^{\low},h_n^{\high}\ra$,
    is a pair of functions, which satisfy the following properties:
    \begin{itemize}
        \item $\dom(h_n^{\low}) = \dom(h_n^{\high}) = A_n^p$,
        \item for every $\nu \in {A_n}^p$, 
    \[ h_n^{\low}(\nu) \in \coll(\rho_n^{\nu},<\tau_n^{\nu})\]
    where $\rho_n^{\nu} = \pi^n_{\mc(a_n^p),0}(\nu)$ and  $\tau_n^\nu = g_n(\rho_n^{\nu})$, 
    and 
        \[ h_n^{\high}(\nu) \in \coll\left( (\tau_n^{\nu})^+, <\lambda_{n+1} \right)\]
    \end{itemize}
    
    \item $h_{-1}^{\high}$ belongs to  $\coll(\omega,<\rho_0^p)$ if $\ell^p \geq 1$, and to $\coll(\omega,<\lambda_0)$ otherwise
    
    \item 
    $h_{\ell^p-1}^{\high} \in V_{\rho_{\ell^p}^{\nu}}$ for every $\nu \in {A_{\ell^p}}^p$, and 
    $h_{n-1}^{\high}(\nu') \in V_{\rho_n^{\nu}}$ for every $n > \ell^p$, $\nu' \in A_{n-1}^p$, and $\nu \in {A_n}^p$.

    \end{enumerate}


A condition $q^* = \la p^*,h^*\ra$ is a direct extension of $p = \la p,h\ra$ if the following conditions hold:
\begin{enumerate}
    \item 
$p^* \geq^* p$ in the sense of $\po$, 
\item for every $n < \ell^p$, 
$h_n^{\low} \subseteq (h_n^*)^{\low}$ and
$h_n^{\high} \subseteq (h_n^*)^{\high}$ ,

\item for every $n \geq \ell^p$, and $\nu \in A_n^{p^*}$,
$h_n^{\low}( \pi^n_{\mc(a_n^{p^*}),\mc(a_n^{p})}(\nu)) \subseteq (h_n^*)^{\low}(\nu)$, and 
$h_n^{\high}( \pi^n_{\mc(a_n^{p^*}),\mc(a_n^{p})}(\nu)) \subseteq (h_n^*)^{\high}(\nu)$.
\end{enumerate}

Given a condition $q = \la p,h\ra$ and an ordinal $\nu \in A_{\ell^p}^p$, we define the one-point end-extension of $q$ by $\nu$, denoted $q \fr \la \nu\ra$ to be the condition $\la p',h'\ra$ given as follows:
\begin{itemize}
\item $p' = p \fr \la \nu\ra$ in the sense of $\po$, in particular $\ell^{p'} = \ell^p + 1$,
\item 
$h'_n = h_n$ for every $n \leq \ell^p$, and in addition, $(h'_{\ell^p})^{\high}$ is now considered as a condition of the restricted collapse poset 
$\coll( \rho_{\ell^p-1}^p,<\tau_{\ell^p}^{\nu})$ (replacing  $\coll( \rho_{\ell^p-1}^p,<\lambda_{\ell^p})$).

\item $(h'_{\ell^{p}})^{\low} = h_{\ell^{p}}^{\low}(\nu)$ and 
$(h'_{\ell^{p}})^{\high} = h_{\ell^{p}}^{\high}(\nu)$,
\item $h'_n = h_n$ for every $n \geq \ell^p+1$.
\end{itemize}

Given a condition $q = \la p,h\ra$ and a finite sequence $\vec{\nu}= \la \nu_{\ell^p},\dots,\nu_{n-1}\ra \in \prod_{\ell^p \leq i <n}A^p_i$, the end-extension of $q$ by $\vec{\nu}$ is defined by 
\[
q \fr \vec{\nu} = q \fr \la \nu_0\ra \fr \la \nu_1\ra \fr \dots  \fr \la \nu_{n-1}\ra
\]

In general, a condition in $\hat{\po}$ extends $q$ if it is obtain from $q$ by finitely many end-extensions and direct extensions. Equivalently, it is a direct extension of $q \fr \vec{\nu}$ for some  $n \geq \ell^p$ and
$\vec{\nu} \in \prod_{\ell^p \leq i <n}A^p_i$ .

${}$\\
\noindent
Let $G \subseteq \hat{\po}$ be a $V$-generic filter.
Through its projection to $\po$, given by $q = (p,h) \mapsto p$, $p \in \po$, it is clear that $V[G]$ adds a sequence of functions $\la t_\al \mid \al < \lam^+\ra$ in the product $\prod_n \tau_n$, where $\tau_n = g_n(\rho_n)$ is derived from the generic filter, as in the case of $\po$. In addition, it is clear that the 
 cardinals in the intervals 
$(\tau_{n-1}^+,\rho_n) \cup (\rho_n, \tau_n)$, $n < \omega$, are all collapsed in $V[G]$.

A standard argument for diagonal Prikry-type forcings with collapses (e.g., see \cite{Gitik-HB}) shows that no other cardinals are collapsed. It follows that  $\tau_n = \aleph_{3n+2}^{V[G]}$ for every $n \in \omega$.
Most relevant to us, is Lemma \ref{lem:StrongPrikryPropertywithCOl} below, which is the $\hat{\po}$ analog of the Strong Prikry Property from \ref{lemma-meetdenseset}.

We first introduce certain useful notations. 
For a sequence of regular cardinals $\vec{\rho} = \la \rho_i \mid i < n\ra$,
satisfying  $\lambda_{i-1} < \rho_i < g_i(\rho_i) < \lambda_i$, 
$\qo_{\vec{\rho}}$ to be the product of collapse forcings
\[
\coll(\omega,<\rho_0) \times 
\left( \prod_{i<n-1} \coll(\rho_i,< g_i(\rho_i)) \times
\coll(g_i(\rho_i)^+, < \rho_{i+1}))
\right) \times \Coll(g_{n-1}(\rho_{n-1})^+,<\lambda_{n})
\]
where for $i = 0$ we set $g_{-1}(\rho_{-1}) = \omega$.
For a condition $q  = \la p,h\ra \in \hat{\po}$ we denote
$\vec{\rho}^q = \la f_n^p(0) \mid n < \ell^p\ra$, 
$\qo_p = \qo_{\vec{\rho}^p}$, and define 
for every $q \geq p$, the collapse restriction 
$q\uhr \qo_{p}$ to be  
$h\uhr \ell^p+1 = \la h_0,\dots h_{\ell^p}\ra$.

\begin{lemma}\label{lem:StrongPrikryPropertywithCOl}
Suppose that $D \subseteq \hat{\po}$ is a dense open set and $q = \la p,h\ra \in \hat{\po}$ a condition. Then there exists a direct extension $q^* \geq^* q$, $n < \omega$,
such that for every $\vec{\nu} \in \prod_{\ell^p \leq i < n}A^{q^*}_i$,
$q^* \fr \vec{\nu}$ reduces meeting $D$ to $\qo_{p^* \vec{\nu}}$, in the sense that there exists a dense open subset 
$D(\vec{\nu})$ of $\qo_{q^* \vec{\nu}}$ so that for every $q' \geq q$, 
if $q'\uhr \qo_{q^* \vec{\nu}} \in D(\vec{\nu})$ then $q' \in D$.
\end{lemma}

This version of the strong Prikry Property naturally extends to versions of
Lemmas \ref{lemma-densesets} and \ref{lemma-functionname}, in which in addition to $n^D,N^D$ ($n^{\vec{\mu}}$, $N^{\vec{\mu}}$, respectively) 
which were used to determine the length of sequences $\vec{\nu}^1 \fr \vec{\nu}^2$ for $q^* \fr \vec{\nu}^1 \fr \vec{\nu}^2$ to meet dense open sets $D$ (or decide values of functions $\name{F}(\vec{\mu})$), here 
an additional function $\bar{D}$ mapping sequences $ \vec{\nu}^1 \fr \vec{\nu}^2$ to dense open subsets
$\bar{D}(\vec{\nu}^1 \fr \vec{\nu}^2)$ of
$\qo_{q^* \fr \vec{\nu}^1 \fr \vec{\nu}^2}$, are added, and reduce the problem of finding 
$q' \geq q^* \fr \vec{\nu}^1 \fr \vec{\nu}^2$ inside $D$, (or deciding $\name{F}(\vec{\mu})$) to $q' \uhr \qo_{q^* \fr \vec{\nu}^1 \fr \vec{\nu}^2}$ being a member of
$\bar{D}(\vec{\nu}^1 \fr \vec{\nu}^2)$.

We note that in particular, the conclusion of \ref{rmk:properness} applies to $\hat{\po}$, since $V_\lambda \subseteq M$ implies that the finite collapse products $\qo_{\vec{\rho}}$ are contained in $M$.\\

From this point, it is straightforward to verify that the rest of the argument, leading to an analogous proof of Theorem \ref{thm:ABSPnew}, remains essentially the same, with the additional key being that the identity of the
newly introduced collapse products $\qo_{q^* \vec{\nu}}$ and their dense sets $\bar{D}^{\vec{\mu}}(\vec{\nu})$ will be decided by the generic information of a bounded part of the scale, $t_\beta$, $\beta < \alpha = \sup(M \cap\lambda^+)$ for a suitable structure $M$ of size $\lambda$. This information remains independent from higher generic scale functions $t_\delta$, $\delta > \alpha$, which allows one to naturally modify the proof of Theorem 
\ref{thm:ABSPnew}, to conclude the same result. 

\begin{theorem}\label{thm:ABSPnewWithCol}
Let $G \subseteq \hat{\po}$ be a generic filter over $V$.
The Approachable Bounded Subset Property (ABSP) holds in $V[G]$ with respect to the sequence $\la \aleph_{3n+2} \mid n < \omega\ra$.
\end{theorem}


\section{Fine structure and the tree-like scale}\label{sec:finestructure}

\subsection{Successor Cardinals}

Let $\mathcal{M} \models \ZFC^-$ be a premouse such that every countable hull of $\mathcal{M}$ has an $(\omega_1 + 1)$ iteration strategy, $\lambda \in M$ a limit cardinal (in $\mathcal{M}$) of $V$-cofinality $\omega$ (which need not agree with its cofinality in $\mathcal{M}$) such that $\lambda^{+}$ exists in $\mathcal{M}$. 

Note if $\mathcal{N}$ is a premouse and $\alpha \in \mathcal{N}$ is such that $\mathcal{N} \models \alpha \text{ is the largest cardinal}$, then we let $(\alpha^+)^\mathcal{N} = \on \cap \mathcal{N}$.

Let $\vec{\kappa} := \<\kappa_n: n < \omega\>$ be a sequence of $\mathcal{M}$-cardinals cofinal in $\lambda$. We do note asume $\vec{\kappa}$ is in $\mathcal{M}$. Let $\tau_n := (\kappa^+_n)^{\mathcal{M}}$. We will define a sequence in $\prod\limits_{n < \omega} \tau_n$ that is increasing, tree-like, and continuous.

Let $C^{\lambda,\mathcal{M}} := \{ \alpha < (\lambda^{+})^{\mathcal{M}} \vert \mathcal{M}\vert\vert \alpha \prec \mathcal{M} \vert\vert (\lambda^{+})^{\mathcal{M}}\}$. For $\alpha \in C^{\lambda,\mathcal{M}}$ let $\mathcal{M}_\alpha$ be the collapsing level for $\alpha$. Let $n_\alpha$ be minimal such that $\rho^{\mathcal{M}_\alpha}_{n+1} = \lambda$, $p_\alpha := p^{\mathcal{M}_\alpha}_{n_\alpha+1}$, and $w_\alpha := w^{\mathcal{M}_\alpha}_{n_\alpha + 1}$. Let also $F_\alpha$ be the top predicate of $\mathcal{M}_\alpha$.

By Lemma \ref{condens} there exists some $\mathcal{M}^n_\alpha \eextend \mathcal{M}$ such that $\mathcal{C}_0(\mathcal{M}^n_\alpha)$ is isomorphic to $\hull^{\mathcal{M}_\alpha}_{n_\alpha + 1}(\kappa_n \cup \{p_\alpha\})$. 

\begin{equation*}
f^{\vec{\kappa},\mathcal{M}}_\alpha(n) = \begin{cases} (\kappa^+_n)^{\mathcal{M}^n_\alpha} & \{w_\alpha,\lambda\} \in \hull^{\mathcal{M}_\alpha}_{n_\alpha + 1}(\kappa_n \cup \{p_\alpha\}) \\ 0 & \text{otherwise} \end{cases}
\end{equation*}

Note that the above function is non-zero almost everywhere, that is if $\lambda \in \mathcal{C}_0(\mathcal{M}_\alpha)$. This can fail if (and only if) $\mathcal{M}_\alpha$ is active and $\nu^{\mathcal{M}_\alpha} = \lambda$. Such $\alpha$ we will call anomalous. For such $\alpha$ we define:

\begin{equation*}
    f^{\vec{\kappa},\mathcal{M}}_\alpha(n) = \begin{cases} (\kappa^+_n)^{\ult(\mathcal{M};F_\alpha \restr \kappa_n)} & \kappa_n > \mu^{\mathcal{M}_\alpha} \\ 0 & \text{otherwise} \end{cases}
\end{equation*}

By the initial segment there must be some $\gamma < \lambda$ such that the trivial completion of $F_\alpha \restr \kappa_n$ is indexed at $\gamma$. We that it is impossible to have the alternative case as $\kappa_n$ is a cardinal and hence not an index \footnote{It also cannot be type Z. Type Z extenders have a largest generator.}

Note that in either case $\mathcal{M}^n_\alpha$ is the least level of $\mathcal{M}$ over which a surjection from $\kappa_n$ on to the ordinal $f^{\vec{\kappa},\mathcal{M}}_\alpha(n)$ is definable. Hence the ordinal defines the level and vice versa.

\emph{In cases where it is clear which mouse and which sequence of cardinals we are talking about, e.g. for the rest of this subsection, we will omit the superscripts.}

\begin{lemma}
Let $\alpha < \beta$ both in $C$. If $m$ is such that $f_\alpha(m) = f_\beta(m)$ then $f_\alpha(n) = f_\beta(n)$ for all $n \leq m$.
\end{lemma}

\begin{proof}
  Note first that if $f_\beta(m) = 0$, then $f_\beta(n) = 0$ and the same holds for $\alpha$. Let us then consider $f_\beta(m) \neq 0$, it follows that $\mathcal{M}^m_\alpha = \mathcal{M}^m_\beta $. We will start with the assumption that neither $\alpha$ nor $\beta$ are anomalous. In that situation we must have that $w_\alpha \in \hull^{\mathcal{M}_\alpha}_{n_\alpha + 1}(\kappa_n \cup \{p_\alpha\})$. This implies that $p_\alpha$ collapses down to $p_{n_\alpha + 1}(\mathcal{M}^m_\alpha)$. The same, of course, holds for $\beta$. Note we must have $n_\alpha = n_\beta$.
  It follows that
  
  \begin{align*}
      \mathcal{C}_0(\mathcal{M}^n_\alpha) & \cong \hull^{\mathcal{M}^m_\alpha}_{n_\alpha + 1}(\kappa_n \cup \{p_{n_\alpha + 1}(\mathcal{M}^m_\alpha)\}) \\
      & = \hull^{\mathcal{M}^m_\beta}_{n_\beta + 1}(\kappa_n \cup \{p_{n_\beta + 1}(\mathcal{M}^m_\beta)\}) \cong \mathcal{C}_0(\mathcal{M}^n_\beta).  \end{align*}

  This implies $f_\beta(n) = f_\alpha(n)$. Note that $f_\beta(n) = 0$ if and only if $w_{n_\beta + 1}(\mathcal{M}^m_\beta) \notin \hull^{\mathcal{M}^m_\beta}_{n_\beta + 1}(\kappa_n \cup \{p_{n_\beta + 1}(\mathcal{M}^m_\beta)\})$ and similarly for $\alpha$.
  
  Assume then that at least one of $\alpha$ and $\beta$ is anomalous. Let us assume that $\alpha$ is anomalous, the proof for $\beta$ is only notationally different. We will realize that, in fact, both must be anomalous. As types are preserved by taking hulls we must have that both are active type III. As at least one is anomalous we do know that the top extender of $\mathcal{M}^m_\alpha$ has no generators above $\kappa_m$. If then the other were not to be anomalous we must have that $\lambda$ is an element of the appropriate hull. This implies that $\mathcal{C}_0(\mathcal{M}^m_\beta)$ has ordinals and hence generators above $\kappa_n$. Contradiction!
  
  As then both are anomalous and $\mathcal{M}^m_\alpha = \mathcal{M}^m_\beta$, we have $F_\alpha \restr \kappa_m = F_\beta \restr \kappa_m$. From this follows $\mu^{\mathcal{M}_\alpha} = \mu^{\mathcal{M}_\beta}$ and $F_\alpha \restr \kappa_n = F_\beta \restr \kappa_n$. Therefore $f_\alpha(n) = f_\beta(n)$.
\end{proof}

\begin{lemma}
  Let $\alpha < \beta$ both in $C$. Then $f_\alpha(n) < f_\beta(n)$ for all but finitely many $n$.
\end{lemma}

\begin{proof}
  Note that since $\alpha < \beta$ are in $C$ then $\mathcal{M}_\alpha \neq \mathcal{M}_\beta$ and so $\mathcal{M}_\alpha \in \mathcal{M}_\beta$.
  Let us first assume that $\beta$ is not anomalous. 
  
  Let $n^*$ be such that $\mathcal{M}_\alpha \in \hull^{\mathcal{M}_\beta}_{n_\beta + 1}(\kappa_{n^*} \cup \{p_\beta\})$. The pre-image of $\mathcal{M}_\alpha$ in $\mathcal{M}^n_\beta$ ($n \geq n^* $) can then compute $\mathcal{M}^n_\alpha$ and hence $f_\alpha(n)$ correctly.
  
  If on the other hand $\beta$ were anomalous, let $n^*$ be such that $\mathcal{M}_\alpha$ is generated by some $a \in \finsubsets{\kappa_{n^*}}$, i.e. $\mathcal{M}_\alpha = \iota_{F_\beta}(h)(a)$ for some $h \in (\fnktsraum{\mu^{\mathcal{M}_\beta}}{\mathcal{M} \vert\vert \mu^{\mathcal{M}_\beta}})$. Then $f_\alpha(n)$ ($n \geq n^*$) can be computed from $\iota_{F_\beta \restr \kappa_n}(h)(a)$ inside $\ult(\mathcal{M};F_\beta \restr \kappa_n)$ by \L o\'{s}'s Theorem. 
\end{proof}

\begin{lemma}
  Let $\beta \in C$ be of uncountable cofinality. Then $\beta$ is a continuity point of the sequence ( i.e.  $f_\beta$ is the exact upper bound of $\<f_\alpha : \alpha \in C \cap \beta\>$).
\end{lemma}

\begin{proof}
  Let $\alpha_n < f_\beta(n)$. We shall find some $\alpha < \beta$ such that $f_\alpha$ dominates $\<\alpha_n: n < \omega\>$ almost everywhere. Towards that end, we deal first with the case where $\beta$ is not anomalous.
  
  For almost all $n < \omega$ we have some surjection from $\kappa_n$ onto $\alpha_n$ in $\mathcal{M}^n_\beta$, given by some parameter $a_n \in \finsubsets{\kappa_n}$ and term $\tau_n$. Let $\xi_n < \rho_{n_\beta}(\mathcal{M}_\beta)$ be such that the image of such a surjection is ($\Sigma_1$)-definable over $\mathcal{M} \vert\vert\xi_n$ with $\Th^{\mathcal{M}_\beta}_{n_\beta}(\xi_n,p_{n_\beta}(\mathcal{M}_\beta))$ as an additional predicate.
  
  By Lemma \ref{fscof}, $\rho_{n_\beta}(\mathcal{M}_\beta)$ has uncountable cofinality. So $\xi := \sup\limits_{n < \omega} \xi_n < \rho_{n_\beta}(\mathcal{M}_\beta)$. Take then some $A$ that codes the $\Sigma_1$ theory of $\mathcal{M} \vert\vert \xi$ with $\Th^{\mathcal{M}_\beta}_{n_\beta}(\xi,p_{n_\beta}(\mathcal{M}_\beta))$ as an additional predicate. Such an $A$ exists in $\mathcal{M}_\beta$. 
  
  Pick some $\alpha < \beta$ such that $A \in \mathcal{M}_\alpha$. Let $n < \omega$ be such that $A$ has a pre-image $\bar{A}$ in $\mathcal{M}^n_\alpha$. $\mathcal{M}^n_\alpha$ can then compute $\alpha_n$ as the ordertype of \[\{(\gamma,\delta)\vert (k,a_n \concat \<\gamma,\delta\>) \in \bar{A}\}\] where $k$ is the G\"{o}del number of $`` \tau_n(a_n)(\gamma) < \tau_n(a_n)(\delta)''$. Hence $\alpha_n < f_\alpha(n)$. 
  Similarly, if $\alpha$ were to be anomalous, we can pick $n$ such that $A = \iota_{F_\alpha}(h)(a)$ for some $h \in \fnktsraum{\mu^{\mathcal{M}_\alpha}}{\mathcal{M}\vert\vert\mu^{\mathcal{M}_\alpha}}$ and $a \in \finsubsets{\kappa_n}$. The rest of the argument remains the same.
  
  Let us then assume that $\beta$ is anomalous. Pick $h_n \in \fnktsraum{\mu^{\mathcal{M}_\beta}}{\mathcal{M}\vert\vert\mu^{\mathcal{M}_\beta}}$ such that $\iota_{F_\beta}(h_n)(a_n)$ is a surjection from $\kappa_n$ onto $\alpha_n$ for some $a_n \in \finsubsets{\kappa_n}$. We have that $\cof((\mu^{\mathcal{M}_\beta,+})^\mathcal{M}) > \omega$.
  
  Pick then some $\xi < (\mu^{\mathcal{M}_\beta,+})^\mathcal{M}$ such that $\<h_n: n < \omega\> \subset \mathcal{M} \vert\vert\xi$. By weak amenability the extender fragment $\bar{F} := \{(a,X) \in F_\beta\vert X \in \mathcal{M} \vert\vert\xi, a \in \finsubsets{\lambda}\}$ in $\mathcal{M}_\beta$. Pick then $\alpha < \beta$ with $\bar{F} \in \mathcal{M}_\alpha$. Any $\mathcal{M}^n_\alpha$ containing $\bar{\bar{F}}$ a pre-image of $\bar{F}$ can then compute $\alpha_n$ as the ordertype of $\{(\gamma,\delta)\vert B^{\gamma,\delta}_n \in \bar{\bar{F}}\}$ where   \[B^{\gamma,\delta}_n =  \{\bar{a} \in \left[ \mu^{\mathcal{M}_\beta}\right]^{\vert b^{\gamma,\delta}_n\vert}\vert h^{a_n,b^{\gamma,\delta}_n}_n(\bar{a})(\id^{\gamma,b^{\gamma,\delta}_n}(\bar{a})) < h^{a_n,b^{\gamma,\delta}_n}_n(\bar{a})(\id^{\delta,b^{\gamma,\delta}_n}(\bar{a}))\},\] and $b^{\gamma,\delta}_n := a \cup \{\gamma,\delta\}$. Hence $\alpha_n < f_\alpha(n)$.
  \end{proof}

\begin{lemma}
   Assume $\<\kappa_n: n < \omega\> \in \mathcal{M}$, then $\<f_\alpha : \alpha \in C\>$ is a scale in $\prod\limits_{n < \omega} \tau_n \cap \mathcal{M}$.
\end{lemma}

\begin{proof}
   Let $f \in \prod\limits_{n < \omega} (\tau_n \slash J_{bd}) \cap \mathcal{M}$. Pick $\alpha \in C$ such that $f \in \mathcal{M}_\alpha$. Then $f(n) < f_\alpha(n)$ for all but finitely many $n$.
\end{proof}

\begin{remark}
  We note that it is possible to associate a sequence in $\prod\limits_{n < \omega} \tau_n$ to any initial segment of $\mathcal{M}$ projecting to $\lambda$ and it would obey the established rules. 
\end{remark}

In certain situations we will want to consider a variant construction. Let us consider an additional set of parameters $\vec{\alpha} := \<\alpha_n: n < \omega\> \in \prod\limits_{n < \omega} \tau_n$. Let $\beta \in C.$ By the condensation lemma there exists some $\mathcal{M}^{n,\alpha_n}_\beta$ such that $\mathcal{C}_0(\mathcal{M}^{n,\alpha_n}_\beta)$ is isomorphic to $\hull^{\mathcal{M}_\beta}_{n_\beta + 1}(\kappa_n \cup \{p_\beta\concat \la \alpha_n\ra\})$. We then define:
\begin{equation*}
    f^{\vec{\kappa},\vec{\alpha},\mathcal{M}}_\beta(n) = \begin{cases} (\kappa^+_n)^{\mathcal{M}^{n,\alpha_n}_\beta} & \{\lambda,w_\beta\} \in \hull^{\mathcal{M}_\beta}_{n_\beta + 1}(\kappa_n \cup \{p_\beta \concat\la \alpha_n\ra\}) \\ 0 & \text{otherwise} \end{cases}
\end{equation*}
If $\beta$ is anomalous, then we use $F_\beta \restr (\alpha_n + 1)$ (instead of $F_\beta \restr \kappa_n$) to define the sequence.

This sequence will behave just like the previously defined sequence. The proofs are mostly the same. The only minor problem adapting these arguments lie in the preservation of standard parameters. Let $p^n_\beta$ be the image of $p_\beta$ under the collapse map in $\mathcal{M}^{n,\alpha_n}_\beta$. Then $p^n_\beta$ might fail to be the standard parameter of $\mathcal{M}^{n,\alpha_n}_\beta$ as it can fail to be \textit{a} good parameter.

Though certainly we do know that $p^n_\beta \concat \la\alpha_n\ra$ is a parameter so the standard parameter is below it in the lexicographic order. As we do have a preimage of the solidity witness in $\mathcal{M}^{n,\alpha_n}_\beta$, its standard parameter can only be lesser on that last component, i.e. $p_{n_\beta + 1}(\mathcal{M}^{n,\alpha_n}_\beta) = p^n_\beta \concat \alpha'$ with $\alpha' \leq \alpha_n$.

Then $\mathcal{M}^{m,\alpha_m}_\beta$ can always compute $\mathcal{M}^{n,\alpha_n}_\beta$ from its standard parameter and the ordinal $\alpha_n$ in a consistent matter, guaranteeing tree-likeness of the sequence. Everything else goes through with minor changes.

\subsection{Limit cardinals}

Let now each of the $\kappa_n$ be an inaccessible cardinal in $\mathcal{M}$. We want to extract from $\mathcal{M}_\beta$ ,$\beta \in C$, a sequence of structures that singularize some $g_\beta(n) < \kappa_n$. For this we need a vector of parameters $\vec{\alpha} = \<\alpha_n: n < \omega\>$ where $\alpha_n < \kappa_n$. We also do require that $\sup\limits_{n < \omega} \alpha_n = \lambda$. When do these parameters give rise to the right structure? This will depend on whether $\beta$ is anomalous or not. When begin with listing three key factors for the case $\beta$ is not anomalous:

\begin{itemize}
   \item[$(1)^\beta_n$] $\sup(\hull^{\mathcal{M}_\beta}_{n_\beta + 1}(\alpha_n \cup \{p_\beta\}) \cap \kappa_n) > \alpha_n$;
   \item[$(2)^\beta_n$] $\kappa_n \in \hull^{\mathcal{M}_\beta}_{n_\beta + 1}(\alpha_n \cup \{p_\beta\})$
   \item[$(3)^\beta_n$] $\hull^{\mathcal{M}_\beta}_{n_\beta + 1}(\alpha_n \cup \{p_\beta\})$ is cofinal in $\rho_{n_\beta}(\mathcal{M}_\beta)$.
\end{itemize}

If $\beta$ is anomalous, we have the following two considerations:

\begin{itemize}
    \item[$(4)^\beta_n$] $\sup(\{\iota_{F_\beta}(h)(a)\vert h \in \fnktsraum{\mu^{\mathcal{M}_\beta}}{(\mu^{\mathcal{M}_\beta})}, a \in \finsubsets{\alpha_n}\} \cap \kappa_n) > \alpha_n$;
    \item[$(5)^\beta_n$] $\kappa_n = \iota_{F_\beta}(h)(a)$ for some $h \in \fnktsraum{\mu^{\mathcal{M}_\beta}}{(\mu^{\mathcal{M}\beta})}$ and $a \in \finsubsets{\alpha_n}$.
\end{itemize}

We say $\beta$ is adequate iff $(1)^\beta_n + (2)^\beta_n + (3)^\beta_n$ or $(4)^\beta_n + (5)^\beta_n$ (depending on type)  are met for all but finitely many $n$. If $\beta$ is adequate and not anomalous then
\begin{equation*}
     g^{\vec{\kappa},\vec{\alpha},\mathcal{M}}_\beta(n) := \begin{cases} \sup(\hull^{\mathcal{M}_\beta}_{n_\beta + 1}(\alpha_n \cup \{p_\beta\}) \cap \kappa_n) & (1)^\beta_m + (2)^\beta_m + (3)^\beta_m \forall m \geq n \\ 0 & \text{ otherwise} \end{cases}
\end{equation*}

We then let $\mathcal{M}^n_\beta$ be the unique initial segment of $\mathcal{M}$ such that $\mathcal{C}_0(\mathcal{M}^n_\beta)$ is isomorphic to $\hull^{\mathcal{M}_\beta}_{n_\beta + 1}(g_\beta(n) \cup \{p_\beta\})$. (Note that the second case in the condensation lemma cannot hold as $g_\beta(n)$ is a limit of cardinals and hence a cardinal itself. This follows by elementarity, trivially so when $n_\beta > 0$ otherwise by $(3)^\beta_n$.)  

If on the other hand $\beta$ is anomalous then
\begin{equation*}
    g^{\vec{\kappa},\vec{\alpha},\mathcal{M}}_\beta(n) := \begin{cases} \sup(\{\iota_{F_\beta}(h)(a)\vert h \in \fnktsraum{\mu^{\mathcal{M}_\beta}}{\mu^{\mathcal{M}_\beta}}, a \in \finsubsets{\alpha_n}\} \cap \kappa_n) & (4)^\beta_m + (5)^\beta_m \forall m \geq n \\ 0 & \text{otherwise} \end{cases}
\end{equation*}

$\mathcal{M}^n_\beta$ will be the unique initial segment of $\mathcal{M}$ with the trivial completion of $F_\beta  \restr g_\beta(n)$ as its top extender. 
As in the previous section, we will omit superscripts for the remainder of this section.

To ensure tree-likeness for this sequence we need a strong interdependence between the ordinal $g_\beta(n)$ and structure $\mathcal{M}^n_\beta$. Towards that end notice that $g_\beta(n)$ is definably singularized over $\mathcal{M}^n_\beta$. The next lemma will show that $\mathcal{M}^n_\beta$ is the least level of $\mathcal{M}$ with this property.

\begin{lemma}
  $g_\beta(n)$ is regular in $\mathcal{M}^n_\alpha$ for all $n$ such that $(1)^\beta_n + (2)^\beta_n + (3)^\beta_n$ or $(4)^\beta_n + (5)^\beta_n$ holds.
\end{lemma}

\begin{proof}
  First we will consider $\beta$ that is not anomalous. Since $\kappa_n$ is regular in $\mathcal{M}_\beta$, it will then be enough to show that $\sup(\hull^{\mathcal{M}_\beta}_{n_\beta + 1}(g_\beta(n) \cup \{p_\beta\}) \cap \kappa_n) = g_\beta(n)$.
  
  Let $\xi < \kappa_n$ be such that $\xi \in \hull^{\mathcal{M}_\beta}_{n_\beta + 1}(g_\beta(n) \cup \{p_\beta\})$. We can then take $\gamma < g_\beta(n)$ and $\delta < \rho_{n_\beta}(\mathcal{M}_\beta)$ such that $\xi \in \hull^{\mathcal{N}_\delta}_1(\gamma)$ where $\mathcal{N}_\delta$ is $\mathcal{M}\vert\vert\delta$ together with $\Th^{\mathcal{M}_\beta}_{n_\beta}(\delta,p_{n_\beta}(\mathcal{M}_\beta))$ as an additional predicate. We can take $\gamma$ and $\delta$ to be in $\hull^{\mathcal{M}_\beta}_{n_\beta + 1}(\alpha_n \cup \{p_\beta\})$ (by definition of $g_\beta(n)$ and $(3)_n$ respectively).
  
  Then $\eta := \sup(\hull^{\mathcal{N}_\delta}_1(\gamma) \cap \kappa_n)$ is also in that hull (uses $(2)_n$) and thus $\xi < \eta < g_\beta(n)$.
  
  Now consider an anomalous $\beta$. We will show that $g_\beta(n)$ is regular in $\ult(\mathcal{M};F_\beta \restr g_\beta(n))$. We have some $h \in \fnktsraum{\mu^{\mathcal{M}_\beta}}{\mu^{\mathcal{M}_\beta}}$ and $a \in \finsubsets{\alpha_n}$ such that $\kappa_n = \iota_{F_\beta}(h)(a)$. We will show that $g_\beta(n) = \iota_{F_\beta \restr g_\alpha(n)}(h)(a)$. As this pair represents a regular cardinal in the larger ultrapower this will suffice.
  
  Pick then some $h_0$ and $b$ such that $b \in \finsubsets{g_\beta(n)}$ and $\iota_{F_\alpha \restr g_\beta(n)}(h_0)(b) < \iota_{F_\beta \restr g_\beta(n)}(h)(a)$. Pick some $c \in \finsubsets{\alpha_n}$ (w.l.o.g. $a \subset c$) and $h_1$ such that $b \subset \iota_{F_\beta}(h_1)(c)$. Define $h_2: \left[\mu^{\mathcal{M}_\beta}\right]^{\vert c\vert} \rightarrow \mu^{\mathcal{M}_\beta}$ by $d \mapsto \sup\{h_0(e) \vert e \in \left[h_1(d)\right]^{\vert b \vert}, h_0(e) < h^{a,c}(d)\}$.
  
  We then have $\iota_{F_\beta \restr g_\beta(n)}(h_0)(b) \leq \iota_{F_\beta \restr g_\beta(n)}(h_2)(c) < g_\beta(n)$ as required.
\end{proof}

Thus $g_\beta(n)$ is regular in $\mathcal{M}^n_\beta$ but is definably singular over it. Thus it is uniquely determined as a level of $\mathcal{M}$ by $g_\beta(n)$.

The following is a straightforward corollary of the proof of the previous Lemma. 

\begin{corollary}\label{stabilitycor}
 Let $\alpha \in C$ be such that $g_\beta(n)$ is defined for all but finitely many $n$. Let $\vec{\alpha}^* := \<\alpha^*_n: n < \omega\>$ be such that $\alpha_n \leq \alpha^*_n < g_\alpha(n)$ for all but finitely many $n$. Then $g^{\vec{\kappa},\mathcal{M},\vec{\alpha}^*}_\beta$ is defined and agrees with $g_\beta$ almost everywhere.
 \end{corollary}

\begin{lemma}
   Say $\beta^*$ is adequate, then every $\beta > \beta*$ in $C$ of uncountable cofinality is also adequate.
\end{lemma}

\begin{proof}
  Let us assume for simplicity's sake that $(1)^{\beta^*}_n + (2)^{\beta^*}_n + (3)^{\beta^*}_n$ holds for all $n$. Let then $n^*$ be such that $\beta^* \in \hull^{\mathcal{M}_\beta}_{n_\beta + 1}(\alpha_m \cup \{p_\beta\})$ for all $m \geq n^*$. Then that hull can compute $\hull^{\mathcal{M}_{\beta^*}}_{n_{\beta^*} + 1}(\alpha_m \cup \{p_{\beta^*}\})$ for all $m \geq n^*$. $(1)^\beta_m + (2)^\beta_m$ then follows.
  
  That is if $\beta^*$ is not anomalous. If it is anomalous note that $\hull^{\mathcal{M}_\beta}_{n_\beta +1}(\alpha_m \cup \{p_\beta\})$ has access to the extender $F_{\beta^*}$ and can compute $\kappa_m$ from it assuming $\alpha_m > \mu^{\mathcal{M}_{\beta^*}}$. $(1)^\beta_m$ follows for similar reasons.
  
  $(3)^\beta_m$ almost everywhere follows for cofinality reasons.
  
  If $\beta$ is anomalous then we take some $h$ and $a \in \finsubsets{\alpha_m}$ such that $\beta^* = \iota_{F_\beta}(h)(a)$ and let $\tau$ be some $r\Sigma_{n_{\beta^*} + 1}$-term such that $\kappa_m = \tau^{\mathcal{M}_{\beta^*}}_m(b_m,p_{\beta^*})$ for $b_m \in \finsubsets{\alpha_m}$. Define then $h^m_0:\left[\mu^{\mathcal{M}_\beta}\right]^{\vert a \cup b_m\vert} \rightarrow \mu^{\mathcal{M}_\beta}$ by \[c  \mapsto \tau^{\mathcal{M}_{h^{a,a\cup b_m}(c)}}_m(\id^{b_m,a \cup b_m}(c),p_{h^{a,a \cup b_m}(c)}).\] We then have $\iota_{F_\beta}(h_0)(a \cup b_m) = \kappa_m$. $(4)^\beta_m$ follows for similar reasons.
  
  The idea is similar if $\beta^*$ and $\beta$ are both anomalous. (Pick $h,a$ representing $F_{\alpha^*}$ etc.) We skip further details.
\end{proof}

Assuming the existence of an adequate ordinal $\beta^*$ we can then show that $\<g_\beta : \beta \in C \backslash \beta^* \cap \cof(\groesser\omega)\>$ is increasing (mod finite), tree-like, and continuous as before.

\begin{lemma}
  Let $\<\kappa_n : n < \omega\>$ and $\<\alpha_n : n < \omega\>$ both in $\mathcal{M}$ then there exists an adequate $\beta^*$ and $\<g_\beta : \beta \in C \backslash \beta^* \cap \cof(\groesser\omega)\>$ is a scale in $\prod\limits_{n < \omega} \kappa_n \cap \mathcal{M}$.
\end{lemma}

\begin{proof}
Any $\beta^*$ of uncountable cofinality (in $\mathcal{M}$) such that $\mathcal{M} \vert\vert \beta^*$ contains both sequences will be adequate. The rest is then as before.
\end{proof}

Note that while we have only considered sequences of a ``pure" type, we could easily deal with sequences $\<\kappa_n: n < \omega\>$ of regular cardinals with both successor cardinals and inaccessible cardinals by mixing both constructions using parameters where needed. With this we finish the proof of Theorem \ref{thmtwo}.

\begin{remark}
  Assuming that $\lambda$ is not subcompact in $\mathcal{M}$ the sequences we defined should be very good, but we have yet to check this in detail. The proof would presumably proceed along similar lines as in \cite{DJS}.
\end{remark}

\section{Core models and the tree-like scale}\label{sec:CMandTLS}

We now want to consider the question when the sequences constructed in the previous section are scales in $V$. For this we need to consider the right mouse. The natural candidate is, of course, the core model. But even core model sequences are not always scales.

To keep the following as accessible as possible we are going to operate under a smallness assumption. This will allow us to cover all known anti tree-like scale results while greatly simplify the following arguments.

This assumption is:
\begin{equation}\label{smallness}
\begin{split}
   \text{There is no inner model } W &\text{ and } F \in W \text{ a total extender } \\
   & \text{such that } \gen(F) \geq (\crit(F)^{++})^W
\end{split}
\end{equation}

\begin{corollary}\label{smallnesskor}
   There is no $\omega_1$-iterable premouse  $(\mathcal{M},\in,\vec{E},F)$ such that $\gen(F) > (\crit(F)^{++})^\mathcal{M}$.
\end{corollary}

\begin{proof}
  Assume towards a contradiction that $(\mathcal{M};\in,\vec{E},F)$ is a counterexample. Then we can generate an inner model $W$ by iterating the top extender out of the universe. Note that by a standard reflection argument, $\omega_1$-iterability is enough to ensure that this model is wellfounded. By the initial segment condition $F \restr (\crit(F)^{++})^W$ is then in $W$ contradicting (\ref{smallness}).
\end{proof}

The reader should be aware, though, that our main results will hold under much weaker anti-Large Cardinals assumptions (up to one Woodin cardinal and beyond). Neither should the choice of indexing scheme affect their validity (though we have yet to check this in detail).

The most immediate payoff of (\ref{smallness}) will be that all iterations we are going to consider are linear (This is one instance in which ms-indexing will make things simpler for us).

\begin{proposition}
   Let $\mathcal{M}$ be a $\omega_1$-iterable premouse, and $\mathcal{T}$ a normal iteration tree on $\mathcal{M}$. Then no $\alpha < \beta \leq \length(\mathcal{T})$ is such that $\crit(E^\mathcal{T}_\beta) < \gen(E^\mathcal{T}_\alpha)$.
\end{proposition}

\begin{proof}
  Let $\alpha < \beta$ such that $\crit(E^\mathcal{T}_\beta) < \gen(E^\mathcal{T}_\alpha)$. There are three cases:
  
  \textbf{Case 1:} $\crit(E^\mathcal{T}_\alpha) < \crit(E^\mathcal{T}_\beta)$
  
  By agreement between models in an iteration we have that $\crit(E^\mathcal{T}_\beta)$ is inaccessible in $\mathcal{M}^\mathcal{T}_\alpha \vert\vert \length(E^\mathcal{T}_\alpha)$ and thus $(\crit(E^\mathcal{T}_\alpha)^{++})^{\mathcal{M}^\mathcal{T}_\alpha \vert\vert \length(E^\mathcal{T}_\alpha)} < \crit(E^\mathcal{T}_\beta)$. As $E^\mathcal{T}_\alpha$ has generators above $\crit(E^\mathcal{T}_\beta)$, $\mathcal{M}^\mathcal{T}_\alpha \vert \length(E^\mathcal{T}_\alpha)$ is a counterexample to Corollary \ref{smallnesskor}.
  
  \textbf{Case 2:} $\crit(E^\mathcal{T}_\alpha) > \crit(E^\mathcal{T}_\beta)$
  
  In $\mathcal{M}^\mathcal{T}_\beta$ due to the agreement between models in an iteration, $\length(E^\mathcal{T}_\alpha) > \crit(E^\mathcal{T}_\beta)$ is a cardinal in $\mathcal{M}^\mathcal{T}_\beta$ so by strong acceptability $\crit(E^\mathcal{T}_\alpha)$ is inaccessible in $\mathcal{M}^\mathcal{T}_\beta$ and thus above $(\crit(E^\mathcal{T}_\beta)^{++})^{\mathcal{M}_\beta})$. As $\mathcal{T}$ is a normal iteration $\length(E^\mathcal{T}_\beta) > \length(E^\mathcal{T}_\alpha) > (\crit(E^\mathcal{T}_\alpha)^+)^{\mathcal{M}^\mathcal{T}_\beta}$ and so $\gen(E^\mathcal{T}_\beta)  > \crit(E^\mathcal{T}_\alpha)$ but then $\mathcal{M}^\mathcal{T}_\beta \vert \length(E^\mathcal{T}_\beta)$ is a counterexample to Corollary \ref{smallnesskor}.
  
  \textbf{Case 3:} $\crit(E^\mathcal{T}_\alpha) = \crit(E^\mathcal{T}_\beta)$
  
   Because $\mathcal{T}$ is normal we have $\length(E^\mathcal{T}_\alpha)  < \length(E^\mathcal{T}_\beta)$ but this means that $E^\mathcal{T}_\beta$ must have generators cofinal in $(\crit(E^\mathcal{T}_\beta)^{++})^{\mathcal{M}^\mathcal{T}_\beta}$. Now, let $\gamma$ be the last drop in the interval $\left(\alpha,\beta\right]$ if it exists or $\alpha + 1$ otherwise. We can assume that $\crit(E^\mathcal{T}_{\gamma-1}) \geq \crit(E^\mathcal{T}_\beta)$. $\iota^\mathcal{T}_{\gamma - 1,\beta}(\crit(E^\mathcal{T}_{\gamma-1}))$ is then the critical point of an extender on the $\mathcal{M}^\mathcal{T}_\beta$ sequence and greater than $\crit(E^\mathcal{T}_\beta)$. As $E^\mathcal{T}_\beta$ must be total over $\mathcal{M}^\mathcal{T}_\beta$. Thus we can produce a class size model $W$ containing $E^\mathcal{T}_\beta$ and agreeing with $\mathcal{M}^\mathcal{T}_\beta$ past $(\crit(E^\mathcal{T}_\beta)^{++})^W$ which contradicts (\ref{smallness}). 
  
\end{proof}

Another consequence of (\ref{smallness}) is that the Jensen-Steel core model $K$ exists by \cite{knomble}. Note that by the smallness assumption there can be no anomalous ordinals in $K$. For the following results we will follow the general framework of the proof of weak covering for that model. Before going into the proofs we shall take quick note of the involved objects.

Let $\lambda$ be a singular cardinal of countable cofinality. Let $\vec{\kappa} := \<\kappa_n: n < \omega\>$ be a sequence cofinal in $\lambda$. Let $\tau_n := (\kappa^+_n)^K$. Consider some $X \prec H_\theta$ ($\theta >> \lambda$) and let $\sigma^X : H_X \to X$ be the inverse of the transitive collapse map.  $X$ will need to satisfy certain properties:

\begin{itemize}
    \item certain phalanxes ``lift" through $\sigma^X$
    \item $\card(X) < \lambda$,
    \item $X$ is tight on $\vec{\kappa}$ (and $\<\tau_n: n < \omega\>$), i.e. $X \cap \prod\limits_{n < \omega} \kappa_n$ is cofinal in $\prod\limits_{n < \omega} (X \cap \kappa_n) \slash J_{bd}$,
    \item the collection of $X \prec H_\theta$ with the above three properties is stationary.
\end{itemize}

The first point is quite vague, and we will provide more details where needed in the course of the argument. By \cite{cover} $\omega$-closed $X$ do satisfy the first property, but it seems possible that there are not enough, i.e. stationary many, $X$ with all properties available. In such cases, by \cite{covernoclosure} we do know that for every internally approachable chain $\vec{Y} := \<Y_i : i < \kappa\>$ in $H_\theta$ there exists some $i < \kappa$ of uncountable cofinality such that $Y_i$ satisfies the first property. That it satisfies the other properties should be easy to see.

Let then from now on $X$ be some such set with the required properties. Let $\sigma_X: H_X \rightarrow X$ be an isomorphism where $H_X$ is transitive. Write $K_X := \ptwimg{\sigma^{-1}_X}{K}$, $\lambda_X := \sigma^{-1}_X(\lambda)$, $\kappa^X_n := \sigma^{-1}_X(\kappa_n)$, etc.

As is standard we will compare $K_X$ with $K$, we should have (for our choice of $X$) that the iteration tree on $K_X$ is trivial (this is $(1)_\alpha$ from \cite{cover} or $(1)^i_\alpha$ from \cite{covernoclosure} respectively). Let then $\mathcal{I}_X$ be the iteration tree on $K$ that arises from the co-iteration. We will simplify notation by writing $\mathcal{M}^X_\alpha$ for  $\mathcal{M}^{\mathcal{I}_X}_\alpha$ etc. Let $\zeta_X := \length(\mathcal{I}_X)$ be the length of the iteration. 

\begin{lemma}\label{drop}
   $(\crit(\sigma_X)^{+})^{K_X} < (\crit(\sigma_X)^+)^K$ and if $E^X_0$ is defined then it is not total over $K$.
\end{lemma}

Note that $K_X$ and $K$ agree up to $(\crit(\sigma_X)^+)^{K_X}$ as a result of the condensation lemma. 

\begin{proof}
  Assume towards a contradiction that $(\crit(\sigma_X)^+)^{K_X} = \crit(\sigma_X)^+)^K$. Then $E_{\sigma_X}$ the $(\crit(\sigma_X),\sigma_X(\crit(\sigma_X)))$-extender derived from $\sigma_X$ measures all subsets of its critical point that are in $K$. It also coheres with $K$ by the elementarity of $\sigma_X$. (This is a little bit of a lie. We would actually need to know that all Mitchell-Steel initial segments of $E_{\sigma_X}$ are on the $K$-sequence. But if this fails we could simply apply the argument we are about to give to the least missing segment instead.)
  
  We do know that the phalanx $\<\<K,\ult(K;E_{\sigma_X})\>,\sigma_X(\crit(\sigma_X))\>$ is iterable. This is $(2)_\alpha$ from \cite{cover} or \cite{covernoclosure} where $\crit(\sigma_X) = (\aleph_\alpha)^{K_X}$. (Once again this is something of a lie. We actually have to replace $K$ with an appropriate soundness witness in the above statement, but we can choose $W$ such that it agrees with $K$ past the level we actually care about. Thus this will not make a difference here.)
  
  But then by \cite[8.6]{CMIP} we have that $E_{\sigma_X}$ is on the $K$-sequence. It should be obvious that $\gen(E_{\sigma_X}) = \sigma_X(\crit(\sigma_X))$ and thus $K \vert \length(E_{\sigma_X})$ contradicts Proposition \ref{smallnesskor}.
  
  As for the second part, assume $E^X_0$ is applied to $K$. By the first part, if $\crit(E^X_0) \geq \crit(\sigma_X)$, then we must truncate. If $(\crit(E^X_0)^+)^{K_X} = \crit(\sigma_X)$, then by elementarity $(\crit(E^X_0)^+)^K = \sigma_X(\crit(\sigma_X)$ so we must truncate.
  
  If $\crit(\sigma_X) \geq (\crit(E^X_0)^{++})^{K_X}$ then its generators must be cofinal in $\crit(\sigma_X)$. So, if the strict inequality holds then $E^X_0$ contradicts Corollary \ref{smallnesskor}. A similar argument applies if $\length(E^X_0) > (\crit(\sigma_X)^+)^{K_X}$.
  
  So, we must have that $\crit(\sigma_X) = (\crit(E^X_0)^{++})^{K_X}$. Consider $\mathcal{M}^X_1$. It agrees with $K_X$ up to $(\crit(\sigma_X)^+)^{K_X}$ and that ordinal is a cardinal there. Thus we can apply the extender $E_{\sigma_X}$ to it. The properties of $X$ will guarantee that $\tilde{K} := \ult(\mathcal{M}^X_1;E_{\sigma_X})$ is iterable (similar to the proof of \cite[Lemma 3.13]{cover}).
  
  We have that $K$ and $\tilde{K}$ agree up to $\sup(\ptwimg{\sigma_X}{(\crit(\sigma_X)^+)^{K_X}})$ which lies past $\sigma_X(\crit(\sigma_X))$ their common $\crit(E^X_0)^{++}$, but on the other hand
  \begin{equation*} (\crit(E^X_0)^{+++})^{\tilde{K}} = \sup(\ptwimg{\sigma_X}{(\crit(\sigma_X)^+)^{K_X}}) <
                    \sigma_X((\crit(\sigma_X)^{+}) = (\crit(E^X_0)^{+++})^K \end{equation*}
  as a result of weak covering. 
  
  Consider then $\tilde{E}$ the first extender applied on the $K$ side in the co-iteration of $K$ and $\tilde{K}$. Its index must be above $\sigma_X(\crit(\sigma_X))$ but its critical point cannot be larger than $\crit(E^X_0)$. $\tilde{E}$ on the $K$-sequence then contradicts (\ref{smallness}).
  
\end{proof}

Remember now the sequence $\<\kappa_n : n < \omega\>$ and the sequence of successors $\<\tau_n: n < \omega\>$. The general idea for the following proofs is to find some ordinal $\alpha_X < \lambda^+$ such that the natural scales of the core model at $\alpha_X$ align with the characteristic function of $X$.

From now on we shall assume that $\kappa_n$ is a cutpoint of (the extender sequence of) $K$ and hence $\kappa^X_n$ is a cutpoint of $K_X$. ( $\alpha \in (\mathcal{M}; \in, \vec{E})$ is a cutpoint (of $\vec{E}$) iff whenever $\crit(E_\beta) < \alpha$, then $\length(E_\alpha) < \alpha$ for all $\beta \in \dom(\vec{E})$.)

\begin{lemma}\label{models}
  There exist some $n_X,k_X < \omega$, a sequence of models  $\<\mathcal{N}^X_n: n_X \leq n < \omega\>$, and maps $\<\upsilon^X_{n,m}: n_X \leq n \leq m < \omega\>$ such that:
  \begin{itemize}
      \item $((\kappa^X_n)^+)^{\mathcal{N}^X_n} = \tau^X_n$ and $\mathcal{N}^X_n$ agrees with $K_X$ up to $\tau^X_n$ for all $n \geq n_X$;
      \item $\mathcal{N}^X_n$ is $(k_X + 1)$-sound above $\kappa^X_n$ for all $n \geq n_X$;
      \item $\upsilon^X_{n,m}:\mathcal{C}_0(\mathcal{N}^X_n) \rightarrow \mathcal{C}_0(\mathcal{N}^X_m)$ is $r\Sigma_{k_X + 1}$-elementary for all $m \geq n \geq n_X$;
      \item $\crit(\upsilon^X_{n,m}) \geq \kappa^X_n$ for all $m \geq n \geq n_X$.
  \end{itemize}
\end{lemma}

For our purposes the critical point of the identity will be defined as the ordinals of its domain.

\begin{proof}
  There are a couple of cases.
  
  \textbf{Case 1:} $\mathcal{I}_X$ has no indices below $\lambda$.
  
  In that case, we have $K_X \vert (\lambda^+)^{K_X} \eextend K$. By Lemma \ref{drop} we do know that some $\mathcal{N}' \eextend K$ exists with $(\crit(\sigma_X)^+)^{\mathcal{N}'} = (\crit(\sigma_X)^+)^{K_X}$ and $\rho_\omega(\mathcal{N}') \leq \crit(\sigma_X)$. By assumption we must have $K_X \vert (\lambda^+)^{K_X} \eextend \mathcal{N}'$.
  
  Take then $\mathcal{N}^*$ to be the smallest initial segment of $K$ that end-extends $K_X \vert (\lambda^+)^{K_X}$ such that $\rho_\omega(\mathcal{N}^*) < \lambda_X$. 
  
  Let $k_X$ be minimal such that $\rho_{k_X + 1}(\mathcal{N}^*) < \lambda_X$. Let $n_X$ be minimal such that $\kappa^X_{n_X} \geq \rho_{k_X + 1}(\mathcal{N}^*)$. We then let $\mathcal{N}^X_n := \mathcal{N}^*$ for all $n \geq n_X$, and let $\upsilon^X_{n,m}$ be the identity for all $m \geq n \geq n_X$. As an initial segment of $K$, $\mathcal{N}^*$ is sound so this works.
  
  \textbf{Case 2:} The set $\{\length(E^X_\beta)\vert\beta < \zeta_X\}$ is bounded below $\lambda_X$.
  
  Let $\eta_X < \zeta_X$ be minimal such that $E^X_{\eta_X}$ has length $\groesser \lambda_X$, if it exists. If there is no such ordinal, let $\eta_X = \zeta_X$. We must then have that $\mathcal{M}^X_{\eta_X}$ agrees with $K_X$ past $\lambda_X$. If $\mathcal{M}^X_{\eta_X}$ has some proper initial segment of length greater than $\lambda_X$ projecting below $\lambda_X$, then this is no different from the previous case.
  
  So let us assume that this is not the case. Let $n_X$ be minimal such that $\kappa^X_{n_X}$ the set $\{\length(E^X_\beta)\vert\beta < \eta_X\}$. By Lemma \ref{drop}, $\mathcal{M}^X_{\eta_X}$ is not a weasel and is $(k_X + 1)$-sound above $\kappa^X_{n_X}$ for some unique $k_X$. 
  
  We then let $\mathcal{N}^X_n := \mathcal{M}^X_{\eta_X}$ for all $n \geq n_X$, and $\upsilon^X_{n,m}$ the identity for all $m \geq n \geq n_X$.
  
  \textbf{Case 3:} The set $\{\length(E^X_\beta)\vert\beta < \zeta_X\}$ is cofinal below $\lambda_X$.
  
 Let $\eta_X := \sup(\{\beta < \zeta_X \vert \length(E^X_\beta) < \lambda_X\})$. By assumption and Lemma \ref{drop} there is some drop in the interval $\left(0,\eta_X \right)$. Let then $\gamma + 1$ be the last such.
 
 Let $k_X$ be minimal such that $\rho_{k_X + 1}((\mathcal{M}^X_{\gamma + 1})^*) \leq \crit(E^X_\gamma)$. Let $n_X$ be minimal such that $\kappa^X_{n_X} \geq \length(E^X_\gamma)$. Let $\eta^X_n < \eta_X$ be minimal such that $\crit(E^X_{\eta^X_n}) \geq \kappa^X_n$ for $n \geq n_X$.
 
 Let then $\mathcal{N}^X_n := \mathcal{M}^X_{\eta^X_n}$ and $\upsilon^X_{n,m} := \iota^X_{\eta^X_n,\eta^X_m}$ for $m \geq n \geq n_X$. It is easy to see that the maps are as wanted, but it remains to check that $\mathcal{N}^X_n$ is $(k_X + 1)$-sound above $\kappa^X_n$. This is going to be the one critical use of the assumption that $\kappa^X_n$ is a cutpoint.
 
 We have to show that the generators of the iteration up to $\eta^X_n$ are bounded by $\kappa^X_n$. If $\eta^X_n$ is a limit this is obvious as by  choice of $\eta^X_n$ all previous critical points are less than $\kappa^X_n$. So assume that $\eta^X_n = \delta + 1$ and $E^X_\delta$ has a generator $\groessergleich \kappa^X_n$. By the initial segment condition we then have that the trivial completion $G$ of $E^X_\delta \restr \kappa^X_n$ is on the sequence of $K_X$. But we have $\crit(G) = \crit(E^X_\delta) < \kappa^X_n$ and $\length(G) > \kappa^X_n$, contradicting that $\kappa^X_n$ is a cutpoint.
\end{proof}

The covering argument goes through three cases. Thanks to Lemma \ref{drop} we can eliminate one of these cases, we will now see that we can also eliminate the other less than convenient case.

\begin{lemma}
  If $\mathcal{N}^X_n$ for $n \geq n_X$ has a top extender, then $\mu^X_n$, its critical point, is $\groessergleich\kappa^X_n$.
\end{lemma}

\begin{proof}
    Let us first assume that $\mathcal{N}^X_n$ has been constructed according to Case 1 or Case 2. Then $\lambda_X$ is a limit cardinal in $\mathcal{N}^X_n$ and thus by (\ref{smallness}) $\mu^X_n$ cannot be smaller than $\lambda_X$.
    
    If $\mathcal{N}^X_n$ is constructed according to Case 3, then some ordinal $\groessergleich \kappa^X_n$ has to be the critical point of an extender on the $\mathcal{N}^X_n$-sequence. As no overlaps can exist on the $\mathcal{N}^X_n$-sequence, $\mu^X_n \geq \kappa^X_n$ follows.
\end{proof}

\begin{remark}\label{modelsrmk}
   Note that $\mathcal{N}^X_n$ in the notation of \cite{cover} is the mouse $\mathcal{P}_\gamma$ where $\kappa_n = \aleph^{K_X}_\gamma$. Recall that $\mathcal{P}_\gamma$ is the least initial segment (if it exists) of $\mathcal{M}^X_{\delta}$ that defines a subset of $\kappa_n$ not in $K_X$ where $\delta < \zeta_X$ is least such that $\gen(E^X_\delta) > \kappa_n$. In addition, by the preceding lemma $\mathcal{P}_\gamma = \mathcal{Q}_\gamma$, i.e. we are avoiding protomice in this construction.
\end{remark}

Let then $\mathcal{N}_X := \dirlim(\<\mathcal{N}^X_n,\upsilon^X_{n,m}: n_X \leq n \leq m < \omega\>)$ and $\upsilon^X_n: \mathcal{C}_0(\mathcal{N}^X_n) \rightarrow \mathcal{C}_0(\mathcal{N}_X)$ the direct limit map. It should be easy to see that $\mathcal{N}_X$ is wellfounded and that the direct limit maps are $r\Sigma_{k_X + 1}$-elementary as they are generated by an iteration on $K$. But more is true:

\begin{lemma}\label{phalanx}
  The phalanx $((K_X,\mathcal{N}_X),\lambda_X)$ is iterable.
\end{lemma}

\begin{proof}
  We cannot quote \cite{cover} here as it seems a priori possible that the mouse (or weasel) $\mathcal{P}_{\beta}$, where $\lambda_X = \aleph^{K_X}_\beta$, from that proof is not equal to $\mathcal{N}_X$. (This would happen if $(\lambda^+_X)^{K_X}$ is not equal to $(\lambda^+_X)^{\mathcal{N}_X}$.)
  
  Nevertheless, the proof presented in \cite{cover} works just as well with $\mathcal{N}_X$ substituted for $\mathcal{P}_\beta$.
  
  For those readers not content with this answer, we want to point out that there is an easy cheat available to us in this case as (\ref{smallness}) implies that $\lambda_X$ must be a cutpoint in $\mathcal{N}_X$, and hence the iterability of the phalanx reduces to the iterability of $\mathcal{N}_X$. The latter holds as $\mathcal{N}_X$ is an iterate of the core model.
\end{proof}

\begin{theorem}\label{dominikthmone}
  Let $\lambda$ be a singular cardinal of countable cofinality. Let $\vec{\kappa} := \<\kappa_n: n < \omega\>$ be a sequence of $K$-cut points cofinal in $\lambda$. Let $\tau_n := (\kappa^+_n)^K$, then $\prod\limits_{n < \omega} \tau_n$ carries a continuous, tree-like scale.
\end{theorem}

\begin{proof}
  We will show that $\vec{f} := \<f^{\vec{\kappa},K}_\alpha : \alpha \in C^{\lambda,K}\>$ as defined in the last section is that scale. Towards that purpose we need to show that this sequence is cofinal in $\prod\limits_{n < \omega} \tau_n \slash J_{bd}$. Let $g \in \prod\limits_{n < \omega} \tau_n \slash J_{bd}$ be arbitrary. Let $X \prec (H_\theta;\in,K \vert\vert\theta,\vec{f})$ be of good type, as explained at the beginning of this section, with $g \in X$. It will suffice to show that there is some $\alpha_X$ such that $f^{\vec{\kappa},K}_{\alpha_X}(n) = \sup(X \cap \tau_n)$ for all but finitely many $n$.
  
  Let $\<\mathcal{N}^X_n,\upsilon^X_{n,m}: n_X \leq n \leq m < \omega\>$ and $\<\mathcal{N}_X,\upsilon^X_n: n_X \leq n < \omega\>$ be as previously discussed appropriate to our choice of $\vec{\kappa}$ and $X$.
  
  The first step will be to show that we can realize the least level of $K$ to define a surjection onto $\sup(X \cap \tau_n)$ by taking an ultrapower of $\mathcal{N}^X_n$ for $n \geq n_X$. Let $\mathcal{O}^X_n := \ult_{k_X}(\mathcal{N}^X_n; \sigma_X \restr K_X \vert \tau^X_n)$ and $\tilde{\sigma}^X_n$ be the ultrapower map for $n \geq n_X$. ( This ultrapower is formed using equivalence classes $\left[f,a\right]_{\sigma_X}$ where $a \in \finsubsets{\kappa_n}$ and $f$ is a function with domain $\left[\kappa^X_n\right]^{\vert a\vert}$ that is $r\Sigma_{k_X}$-definable over $\mathcal{N}^X_n$.)
  
  We do know that these models are wellfounded, in fact, the phalanx $((K,\mathcal{O}^X_n),\kappa_n)$ must be iterable. (This is $(2)_\beta$ from \cite{cover} or \cite{covernoclosure}, where $\kappa^X_n = \aleph^{K_X}_\beta$.) This means that $\mathcal{O}^X_n$ is an inital segment of $K$. Furthermore, $\mathcal{O}^X_n$ is sound above $\kappa_n$, and $\tilde{\sigma}^X_n(\tau^X_n) = \sup(X \cap \tau_n)$ is a cardinal there by the choice of $\mathcal{N}^X_n$. This means that $\mathcal{O}^X_n$ is the level of $K$ we are looking for.
  
  The next step must be to tie the sequence $\<\mathcal{O}^X_n : n_X \leq n < \omega\>$ to some level of $K$ projecting to $\lambda$. Our candidate is $\mathcal{O}_X := \ult_{k_X}(\mathcal{N}_X; \sigma_X \restr K_X \vert \lambda_X)$. Let $\tilde{\sigma}_X$ be the ultrapower map. By Lemma \ref{phalanx} and the lifting properties of our $X$ not only is $\mathcal{O}_X$ wellfounded, but it is an initial segment of the core model. Let $\alpha_X := (\lambda^+)^{\mathcal{O}_X}$.
  
  The last thing we need are appropriate embeddings from $\mathcal{O}^X_n$ into $\mathcal{O}_X$ for $n \geq n_X$. Define $\pi^X_n:\mathcal{C}_0(\mathcal{O}^X_n) \rightarrow \mathcal{C}_0(\mathcal{O}_X)$: let $\pi^X_n(\left[f,a\right]_{\sigma_X}) = \left[\upsilon^X_n(f) \restr \finsubsets{\kappa^X_n},a\right]_{\sigma_X}$. It is to be understood here that if $f$ is not an element of $\mathcal{C}_0(\mathcal{N}^X_n)$ but merely definable over it, then $\upsilon^X_n(f)$ is the function over $\mathcal{C}_0(\mathcal{N}_X)$ with the same definition and parameters moved according to $\upsilon^X_n$.
  
  Let now $f$ an $r\Sigma_{k_X}$ definable function over $\mathcal{C}_0(\mathcal{N}^X_n)$, $\phi$ an $r\Sigma_{k_X}$-formula, and $a \in \finsubsets{\kappa_n}$.
  
  \begin{align*}
      \mathcal{C}_0(\mathcal{O}^X_n) \models \phi(\left[f,a\right]_{\sigma_X}) & \Leftrightarrow a \in \sigma_X(\{b \in \finsubsets{\kappa^X_n}\vert \mathcal{N}^X_n \models \phi(f(b))\}) \\
                                                                & \Leftrightarrow a \in \sigma_X(\{b \in \finsubsets{\kappa^X_n}\vert \mathcal{N}_X \models \phi(\upsilon^X_n(f)(b))\}) \\
                                                                & \Leftrightarrow \mathcal{C}_0(\mathcal{O}_X) \models \phi(\left[\upsilon^X_n(f) \restr \finsubsets{\kappa^X_n},a\right]_{\sigma_X})
  \end{align*}
  
  This shows that $\pi^X_n$ is $r\Sigma_{k_X}$-elementary. Consider then the following diagram:
  
  
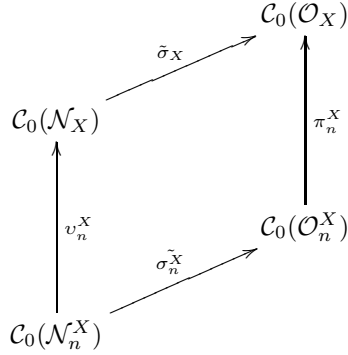
\[
  \xymatrix{                                                                      & & \mathcal{C}_0(\mathcal{O}_X)                     \\
             \mathcal{C}_0(\mathcal{N}_X) \ar[urr]^{\tilde{\sigma}_X}                            & &                                   \\
                                                                                  & & \mathcal{C}_0(\mathcal{O}^X_n) \ar[uu]_{\pi^X_n} \\
             \mathcal{C}_0(\mathcal{N}^X_n) \ar[uu]_{\upsilon^X_n} \ar[urr]^{\tilde{\sigma^X_n}} & &}\]
  
  The diagram commutes, and all of $\upsilon^X_n,\tilde{\sigma_X},\tilde{\sigma^X_n}$ are cofinal (in $\rho_{k_X}(\cdot)$). Thus so is $\pi^X_n$ which shows that it is $r\Sigma_{k_X + 1}$-elementary. Also note that the critical point of $\pi^X_n$ is $\groessergleich\kappa_n$.
  
  It then follows that $\mathcal{C}_0(\mathcal{O}^X_n)$ is isomorphic to $\hull^{\mathcal{O}_X}_{k_X + 1}(\kappa^X_n \cup \{p_{k_X + 1}(\mathcal{O}_X)\})$, so $\sup(X \cap \tau_n) = f^{\vec{\kappa},K}_{\alpha_X}(n)$ for $n \geq n_X$.
\end{proof}

\begin{remark}
   The last line is inaccurate, as it seems possible that $\alpha_X \notin C^{\lambda,K}$ meaning $f^{\vec{\kappa},K}_{\alpha_X}$ might not be defined. Nevertheless the structure $\mathcal{O}^X_n$ is definable from $\alpha_X$ and $\kappa_n$ in $K$ which implies that the sequence $\<\sup(X \cap \tau_n): n_X \leq n < \omega\>$ is dominated by some $f^{\vec{\kappa},K}_\beta$ for $\beta \in C^{\lambda,K}$. 
\end{remark}

\begin{corollary}
   In the above situation $\alpha_X = \sup(X \cap \lambda^+)$.
\end{corollary}

\begin{proof}
   By continuity $f^{\vec{\kappa},K}_{\sup(X \cap \lambda^+)}$ is the exact upper bound of $\<f^{\vec{\kappa},K}_\beta: \beta < \sup(X \cap \lambda^+)\>$. On the other hand, as we know that $\vec{f}$ is a scale, by the tightness of $X$ we also know that $\<\sup(X \cap \tau_n) : n < \omega\>$ is also an exact upper bound for this sequence. This implies that both agree almost everywhere, but the latter equals $f^{\vec{\kappa},K}_{\alpha_X}$ almost everywhere. The desired equality then follows.  
\end{proof}

Let us now move on to the second theorem. This one concerns scales on products that concentrate on ordinals that are inaccessible in $K$. We will see that scales on such ordinals are significantly more restricted.

\begin{theorem}\label{dominikthmtwo}
   Let $\lambda$ be a singular cardinal of countable cofinality. Let $\<\kappa_n:n < \omega\>$ be a cofinal sequence such that each $\kappa_n$ is an inaccessible limit of cutpoints of $K$. Assume there is some $\delta < \lambda$ such that ordinals $\beta$ with $\mitord^K(\beta) \geq \delta$ are bounded in each of the $\kappa_n$. Then $\prod\limits_{n < \omega} \kappa_n $ admits a continuous, tree-like scale.
\end{theorem}

Let from now on $\vec{\kappa} := \<\kappa_n : n < \omega\>$ and $\delta < \lambda$ be as in the statement of the theorem. As this theorem deals with scales on ordinals which are inaccessible in $K$ we will have need of a theorem that provides information about the possible cofinalites of such ordinals. In general, we cannot expect these cofinalities to be high because of the existence of Prikry forcing. The next theorem essentially states that this is the only real obstacle. Versions of this theorem for different forms of the core model have existed for some time, but its newest form appropriate for the Jensen-Steel core model is due to Mitchell and Schimmerling.

\begin{theorem}[Mitchell-Schimmerling]\label{cofinalities}
   Assume there is no inner model with a Woodin cardinal, and let $K$ be the Jensen-Steel core model. Let $\alpha \geq \aleph_2$ be such that $\alpha$ is regular in $K$, but $\cof(\alpha) < \card(\alpha)$. Then $\mitord^K(\alpha) \geq \nu$ where $\cof(\alpha) = \omega \cdot \nu$.
\end{theorem}

See \cite{MitSchim}. Alternatively, as we only deal with linear iterations here it should be plausible that the results from \cite{Sean} even though not directly applicable can be mimicked here to achieve a similar end.

Let us now again consider some $X \prec (H_\theta;\in,\ldots)$ containing relevant objects. In addition to its previous properties we will require that $\cof(\sup(X \cap \kappa_n)) > \delta$. Note then that by our assumption and \ref{cofinalities}, and this fact will be crucial, $\sup(X \cap \kappa_n)$ is a singular cardinal in $K$.

We will once again have need of the directed system $\<\mathcal{N}^X_{n,m}, \upsilon^X_{n,m} : n_X \leq n \leq m < \omega\>$ and its limit $\<\mathcal{N}_X,\upsilon^X_n: n_X \leq n < \omega\>$, but we will require some additional properties.

\begin{lemma}
   There exist $\tilde{\alpha}^X_n < \kappa^X_n$ such that $\hull^{\mathcal{N}^X_n}_{k_X + 1}(\tilde{\alpha}^X_n \cup \{p_{k_X + 1}(\mathcal{N}^X_n)\})$ is cofinal in $\kappa^X_n$.
\end{lemma}

\begin{proof}
  If the system is constructed as in Case 1 and 2 then there is a single $\alpha$ such that $\mathcal{N}^X_n$ (which is independent of $n$) is sound above $\alpha$ so the conclusion follows.
  
  Consider then that the system is constructed as in Case 3. Pick $n \geq n_X$. Recall that $\mathcal{N}^X_n = \mathcal{M}^X_{\eta^X_n}$ and $\gamma + 1$ is the last drop below $\eta^X_n$. Note that by definition of $\eta^X_n$ all critical points before that point are less than $\kappa^X_n$. There are two cases.
  
  \textbf{Case 3.1:} $\eta^X_n = \bar{\eta} + 1$.
  
  In that case as $\kappa^X_n$ is a limit cardinal we must have that $\length(E^X_{\bar{\eta}}) < \kappa^X_n$. Let $\tilde{\alpha}^X_n < \kappa^X_n$ be such that $\mathcal{M}^X_{\eta^X_n}$ is sound above $\tilde{\alpha}^X_n$. 
  
  \textbf{Case 3.2:} $\eta^X_n$ is a limit ordinal.
  
  Let $\gamma < \beta < \eta^X_n$ be such that $\iota^X_{\beta,\eta^X_n}(\bar{\kappa}) = \kappa^X_n$ for some $\bar{\kappa} \in \mathcal{M}^X_\beta$. We must have that $\iota^X_{\beta,\xi}(\bar{\kappa}) \geq \crit(\iota^X_{\xi,\eta^X_n})$ for all $\xi \in \left[\beta,\eta^X_n\right)$. The key is to consider when equality holds in the above equation.
  
  Let us assume towards a contradiction that $\iota^X_{\beta,\xi}(\bar{\kappa}) = \crit(\iota^X_{\xi,\eta^X_n})$ for an unbounded in $\eta^X_n$ set $A$. For $\nu \in \lim(A) \cap \eta^X_n$ we have 
  \begin{equation*}
      \crit(\iota^X_{\nu,\eta^X_n}) \geq \sup\limits_{\xi \in A \cap \nu} \crit(\iota^X_{\xi,\eta^X_n}) = \sup\limits_{\xi \in A \cap \nu} \iota^X_{\beta,\xi}(\bar{\kappa}) = \iota^X_{\beta,\nu}(\bar{\kappa})
  \end{equation*}
  and hence $\nu \in A$. But then $B := \{\iota^X_{\beta,\xi}(\bar{\kappa}) \vert \xi \in A\}$ is a club of indiscernibles in $\kappa^X_n$. As $\sigma_X$ is continuous at points of cofinality $\omega$, $C := \ptwimg{\sigma_X}{B}$ is an $\omega$-club in $\sup(X \cap \kappa_n)$. As the latter was singular there must exist a club $D \subset \sup(X \cap \kappa_n)$ consisting of $K$-singulars. But $C \cap D \neq \emptyset$, and $C$ consists of $K$-regulars. Contradiction!
  
  We conclude that $\crit(\iota^X_{\xi,\eta^X_n}) < \iota^X_{\beta,\xi}(\bar{\kappa})$ for all $\xi \geq \nu$ for some $\nu \in \left[\beta,\eta^X_n\right)$. This means that $\iota^X_{\nu,\eta^X_n}$ is continuous at $\iota^X_{\beta,\nu}(\bar{\kappa})$. We then finish the argument by noticing that $\mathcal{M}^X_\nu$ is $(k_X + 1)$-sound above $\crit(\iota^X_{\nu,\eta^X_n})$, and thus $\hull^{\mathcal{M}^X_{\eta^X_n}}_{k_X + 1}(\crit(\iota^X_{\nu,\eta^X_n}) \cup \{p_{k_X + 1}(\mathcal{M}^X_{\eta^x_n})\})$ is cofinal in $\kappa^X_n$.
  \end{proof}

 \begin{proof}[Proof of Theorem \ref{dominikthmtwo}]
  We want to show that for some $\vec{\alpha}_X := \<\alpha^X_n: n_X \leq n < \omega\> \in X$ the sequence $\<\sup(X \cap \kappa_n): n < \omega\>$ agrees almost everywhere with $g^{\vec{\kappa},K,\vec{\alpha}_X}_{\alpha_X}$. (Implicit here is that $\alpha_X$ will be adequate.) 
  
  Recall the structures $\mathcal{O}^X_n$ from the proof of the preceding theorem. We will need a slightly different structure here. Let $(\mathcal{O}^X_n)^* := \ult_{k_X}(\mathcal{N}^X_n; \sigma_X \restr K_X \vert \kappa^X_n)$. (This ultrapower is formed using equivalence classes $\left[f,a\right]_{\sigma_X}$ where $a \in \finsubsets{\sup(X \cap \kappa_n)}$ and $f$ is a function with domain $\left[\gamma\right]^{\vert a\vert}$ where $\gamma < \kappa_n$ is a cardinal with $a \subset \sigma_X(\gamma)$ and $f$ is $r\Sigma_{k_X}$-definable over $\mathcal{N}^X_n$. Note that functions with different domains can be compared by adding dummy values.)
  
  Let $\bar{\sigma}^X_n$ be the ultrapower map. Note that $\bar{\sigma}^X_n$ maps $\kappa^X_n$ cofinally into $\sup(X \cap \kappa_n)$ so we have $\hull^{(\mathcal{O}^X_n)^*}_{n_X + 1}(\alpha^X_n \cup \{p_{k_X + 1}((\mathcal{O}^X_n)^*)\})$ is cofinal in $\sup(X \cap \kappa_n)$ where $\alpha^X_n := \bar{\sigma}^X_n(\tilde{\alpha}^X_n)$.
  
  The phalanx $((K,(\mathcal{O}^X_n)^*),\sup(X \cap \kappa_n))$ is iterable as $\mathcal{C}_0((\mathcal{O}^X_n)^*)$ can be mapped into $\mathcal{C}_0(\mathcal{O}^X_n)$ by a map with critical point $\sup(X \cap \kappa_n)$, so $(\mathcal{O}^X_n)^*$ is an initial segment of $K$, in fact, the least one to define a witness to the singularity of $\sup(X \cap \kappa^X_n)$.
  
  Just as before we can map $\mathcal{C}_0((\mathcal{O}^X_n)^*)$ into $\mathcal{C}_0(\mathcal{O}_X)$, so $g^{\vec{\kappa},K,\vec{\alpha}_X}_{\alpha_X}(n) = \sup(X \cap \kappa_n)$ for all $n \geq n_X$. We would like to have $\vec{\alpha}_X \in X$. This is obvious if $X$ is $\omega$-closed. If $X$ is merely internally approachable then we can still find some $\vec{\alpha}' \in X \cap \prod\limits_{n < \omega} \sup(X \cap \kappa_n)$ that dominates $\vec{\alpha}_X$ almost everywhere. Then $g^{\vec{\kappa},K,\vec{\alpha}_X}_{\alpha_X}$ and $g^{\vec{\kappa},K,\vec{\alpha}'}_{\alpha_X}$ agree almost everywhere by Corollary \ref{stabilitycor}, so we can replace $\vec{\alpha}_X$ with $\vec{\alpha}'$.
  
  By Fodor's Lemma we then have a stationary set of $X$ and a single $\vec{\alpha}$ such that $g^{\vec{\kappa},K,\vec{\alpha}}_{\alpha_X}$ agrees with $\<\sup(X \cap \kappa_n : n < \omega\>$ almost everywhere. This then shows that $\<g^{\vec{\kappa},K,\vec{\alpha}}_{\alpha} : \alpha \in C \cap \cof(\groesser\omega)\>$ is a scale.
\end{proof}

We are going to finish by showing how to weaken the assumption of Theorem \ref{dominikthmone} yet achieving the same result. It is here that we will make use of the sequence $\<f^{\vec{\kappa},K,\vec{\alpha}}_\alpha: \alpha \in C^{\lambda,K}\>$.

We say a cardinal $\kappa \in K$ is a weak cutpoint if $\crit(E) < \kappa$ implies $\length(E) < (\kappa^+)^K$ for all extenders $E$ on the $K$-sequence.

\begin{theorem}\label{dominikthmthree}
Let $\lambda$ be a singular cardinal of countable cofinality. Let $\<\kappa_n: n < \omega\>$ be a sequence of weak cutpoints cofinal in $\lambda$. Let $\tau_n := (\kappa^+_n)^K$. Then $\prod\limits_{n < \omega} \tau_n $ carries a continuous, tree-like scale.
\end{theorem}

\begin{lemma}\label{modelsvar}
  There exist some $n_X,k_X < \omega$, a sequence of ordinals $\<\tilde{\alpha}^X_n : n_X \leq n < \omega\>$, a sequence of models  $\<\mathcal{N}^X_n: n_X \leq n < \omega\>$, and maps $\<\upsilon^X_{n,m}: n_X \leq n \leq m < \omega\>$ such that:
  \begin{itemize}
      \item $((\kappa^X_n)^+)^{\mathcal{N}^X_n} = \tau^X_n$ and $\mathcal{N}^X_n$ agrees with $K_X$ up to $\tau^X_n$ for all $n \geq n_X$;
      \item $\mathcal{N}^X_n$ is $(k_X + 1)$-sound above $\kappa^X_n$ relative to $p_{k_X + 1}(\mathcal{N}^X_n) \concat \tilde{\alpha}^X_n$ for all $n \geq n_X$;
      \item $\upsilon^X_{n,m}:\mathcal{C}_0(\mathcal{N}^X_n) \rightarrow \mathcal{C}_0(\mathcal{N}^X_m)$ is $r\Sigma_{k_X + 1}$-elementary for all $m \geq n \geq n_X$;
      \item $\crit(\upsilon^X_{n,m}) \geq \max\{\kappa^X_n,\tilde{\alpha}^X_n + 1\}$ for all $m \geq n \geq n_X$.
  \end{itemize}
\end{lemma}

\begin{proof}
  This proof goes through the same cases as the proof of Lemma \ref{models}, in fact, many of the cases will be the same. (In those cases we can take $\tilde{\alpha}^X_n$ to be $0$.) In the interest of time we shall only deal with the case that is unique to this situation.
  
  Let us assume that $\eta_X := \sup(\{\beta < \zeta_X \vert \length(E^X_\beta) < \lambda\})$ is a limit ordinal. Let $\gamma + 1$ be the last drop in the interval $\left(0,\eta_X \right)$. Let $k_X$ be minimal such that $\rho_{k_X + 1}((\mathcal{M}^X_{\gamma_1})^*) \leq \crit(E^X_\gamma)$. Let $n_X$ be minimal such that $\kappa^X_{n_X} \geq \length(E^X_\gamma)$. Let $\eta^X_n < \eta_X$ be minimal such that $\crit(E^X_{\eta^X_n}) \geq \kappa^X_n$ for $n \geq n_X$.
 
 Let then $\mathcal{N}^X_n := \mathcal{M}^X_{\eta^X_n}$ and $\upsilon^X_{n,m} := \iota^X_{\eta^X_n,\eta^X_m}$ for $m \geq n \geq n_X$. Let $n \geq n_X$ be such that $\eta^X_n = \tilde{\eta}^X_n + 1$ and $\gen(E^X_{\tilde{\eta}^X_n}) \geq \kappa_n$. Otherwise the argument will proceed just as in the proof of Lemma \ref{models} (and $\tilde{\alpha}^X_n = 0$).
 
 First note then $\gen(E^X_{\tilde{\eta}^X_n}) < \tau^X_n$ as otherwise $\kappa^X_n$ could not be a weak cutpoint by the initial segment condition. Moreover, it then follows that $E^X_{\tilde{\eta}^X_n}$ has a largest generator as otherwise $\gen(E^X_{\tilde{\eta}^X_n}) \in \left(\kappa^X_n,\tau^X_n\right)$ must be a cardinal in $\mathcal{N}^X_n$.
 
 Let $\tilde{\alpha}^X_n$ be that largest generator. We will be done if we can show that $\kappa^X_n \cup \{\tilde{\alpha}^X_n\}$ generates the whole ultrapower. Let then $\tilde{\mathcal{M}} := \ult(\mathcal{M}^X_{\tilde{\eta}^X_n}; E^X_{\tilde{\eta}^X_n} \restr \kappa^X_n \cup \{\tilde{\alpha}^X_n\})$ and $\tilde{\iota}:\mathcal{C}_0(\tilde{\mathcal{N}}) \rightarrow \mathcal{C}_0(\mathcal{N}^X_n)$ be the canonical embedding.
 
 We have that $\tilde{\alpha}^X_n \in \left(\kappa^X_n,\tau^X_n\right)$ is in the range of $\tilde{\iota}$, thus so is $\kappa^X_n = \card^{\mathcal{N}^X_n}(\tilde{\alpha}^X_n)$ and some surjection from $\kappa^X_n$ on to $\tilde{\alpha}^X_n$. Then $\tilde{\alpha}^X_n \subset \ran{\tilde{\iota}}$ and thus so are all of the other generators of $E^X_{\tilde{\eta}^X_n}$.
\end{proof}

\begin{remark}
     Note that in the ``special" case of Lemma \ref{modelsvar} unlike Remark \ref{modelsrmk} $\mathcal{N}^X_n$ is not equal to the mouse $\mathcal{P}_\gamma$ (where $\kappa_n = \aleph^{K_X}_\gamma$) from \cite{cover}, but it is equal to $\mathcal{Q}_\gamma$. To see this we must first realize that $\mathcal{P}_\gamma$ must be an initial segment of $\mathcal{M}^X_{\tilde{\eta}^X_n}$ as in this case $E^X_{\tilde{\eta}^X_n}$ has generators $\groessergleich\kappa_n$.
     
     As the Dodd projectum of $E^X_{\tilde{\eta}^X_n}$ is below $\kappa_n$ we have, in fact, $\mathcal{P}_\gamma = \mathcal{M}^X_{\tilde{\eta}^X_n} \vert \length(E^X_{\tilde{\eta}^X_n})$. Note though that lifting this mouse by $\sigma_X$ would create a proto mouse. Hence we must move to the mouse $\mathcal{Q}_\gamma$ which is formed by applying the extender $E^X_{\tilde{\eta}^X_n}$ using the usual iteration tree rules. Hence the resulting mouse must be equal to $\mathcal{M}^X_{\tilde{\eta}^X_n + 1} = \mathcal{N}^X_n$.
\end{remark}

\begin{proof}[Proof of Theorem \ref{dominikthmthree}]
Let $\<\mathcal{N}^X_n: n_X \leq n < \omega\>$,$\<\upsilon^X_{n,m}:n_X \leq n \leq m < \omega\>$ and $\<\tilde{\alpha}^X_n:n_X \leq n < \omega\>$ as in the lemma. We will find some $\alpha_X$ and $\vec{\alpha}_X := \<\alpha^X_n: n_X \leq n < \omega\>$ such that $\sup(X \cap \tau_n) = f^{\vec{\kappa},K,\vec{\alpha}}_{\alpha_X}(n)$ for all $n \geq n_X$. In fact, $\alpha^X_n = \sigma_X(\tilde{\alpha}^X_n)$ will do. A priori $\vec{\alpha}_X$ will depend on $X$ but we will be able to deal with that by pressing down just as in the proof of Theorem \ref{dominikthmtwo}.

We can mostly proceed as in the proof of Theorem \ref{dominikthmone}. We will form $\mathcal{O}^X_n$ and $\mathcal{O}_X$ as before, and generate embeddings between them by lifting $\upsilon^X_n$. Note that $\upsilon^X_n$ will not move $\tilde{\alpha}^X_n$ as iteration maps do not move generators. So neither will its lift move $\alpha^X_n$. Thus $\mathcal{C}_0(\mathcal{O}^X_n)$ will be isomorphic to $\hull^{\mathcal{O}_X}_{k_X + 1}(\kappa_n \cup \{p_{k_X + 1}(\mathcal{O}_X)\concat \alpha^X_n\})$ as required.

As before the fact that the phalanx $\<\<K,\mathcal{O}^X_n\>,\kappa_n\>$ is iterable follows from the covering lemma, noticing that in this case we might have to consider the mouse $\mathcal{Q}_\beta$ not $\mathcal{P}_\beta$ as explained above. Fortunately, this does not change anything about the rest of the argument. We skip further detail.
\end{proof}

\begin{proof}[Proof of Theorem \ref{thmthree}]
We have different cases depending on if the $\kappa_i$ are limit cardinals or successor cardinals in the core model. Let us first assume that all $\kappa_i$ share a type. If that shared type is limit cardinals, then we can use Theorem \ref{dominikthmtwo} to finish. If that type is successor cardinals we have two cases: if $\bar{\kappa_i}$ is the $K$-predecessor of $\kappa_i$ is measurable, then it must be a cutpoint by the smallness assumption therefore we can use Theorem \ref{dominikthmone} to finish; if it is not, then it must be a weak cutpoint thus we can use Theorem \ref{dominikthmthree} to finish.

In cases of mixed type, divide the sequence into three parts of pure type. Each of these parts do have a scale by the above. These individual scales can then be integrated. This works as individual elements of the different scales can be tied to some common ordinal $\kleiner\lambda^+$.
\end{proof}

\section{Open questions}\label{sec:open}

We conclude this work with a discussion on further possible developments, and open questions.
\begin{enumerate}

    \item 
    Consider the following natural strengthening of the ABSP with respect sequence of regular cardinals $\vec{\tau} = \la \tau_n \mid n < \omega\ra$ with $\lambda = \cup_n \tau_n$: For every sufficiently large regular cardinal $\theta$ and internally approachable structure $N \elem (H_\theta,\in,\vec{\tau})$, there is some $m < \omega$, so that
    for every strictly increasing sequence $d_0,\dots,d_k \in \omega\setminus m$ and $F \in N$, 
    $F : [\lambda]^k \to \lambda$, if
    \[
    F(\chi_N(\tau_{d_1}),\dots,\chi_N(\tau_{d_k})) < \tau_{d_0}
    \]
then 
\[
    F(\chi_N(\tau_{d_1}),\dots,\chi_N(\tau_{d_k})) \in N.
    \]

Is it consistent?
    
\item  
We saw in Section \ref{sec:DownToAlephOmega} that from the same large cardinal assumptions of Theorem \ref{thm:ABSPnew}, 
it is consistent that ABSP holds with respect to a sequence $\la \tau_n \mid n <\omega$ so that $\tau_n = \aleph_{2n}$ for all $n < \omega$.
Is ABSP consistent with respect to cofinite sequence of the $\aleph_n$'s?

\item  The definitions of Tree-like scales, Essentially Tree-like scales, ASFP, and ABFP naturally extend to uncountable sequences of cardinals 
$\la \tau_i \mid i < \rho\ra$, $\rho > \aleph_0$ regular. 
Are those principles consistent? If so, what is their consistency strength?

\item  Another natural extension of the principles AFSP and ABFP, is to require the appropriate principle to hold for any elementary substrucute $N \elem (H_\theta\in \vec{\tau})$. 
Is it consistent?



\item Is there a version of Theorem \ref{thmtwo} for Neeman-Steel long extender mice?

\item Pereira showed in \cite{Luis2} that it consistent relative to the existence of a supercompact cardinal that there exist products $\prod\limits_{n < \omega} \tau_n$ carrying a continuous tree-like scale of length greater than $\sup(\<\tau_n\>_n)^+$. Can the same be achieved from a weaker large cardinal assumptions at the level of strong cardinals? 

\end{enumerate}

\bibliographystyle{plain}
\bibliography{bibli}

\end{document}